\documentclass[final,leqno]{siamltex}

\def\norm#1{\|#1\|}

\usepackage{color}
\usepackage{amsmath}
\usepackage{amsfonts}
\usepackage{textcomp}
\usepackage{epsfig}
\usepackage{url}
\usepackage{slashbox}

\newcommand{\OM}{\Omega}
\newcommand{\RE}{\mathbb{R}}
\newcommand{\PO}{\mathbb{P}}

\newcommand{\PDif}[2]{\frac{\partial {#1}}{\partial {#2}}}
\newcommand{\Lsp}{\textrm{L}}
\newcommand{\Hsp}{\textrm{H}}
\newcommand{\Tsp}{\textrm{T}}

\newcommand{\Arr}[2]{
	\left(
		\begin{array}{c}
		{#1} 
		\\
		{#2}
		\end{array}
	\right)
}
\newcommand{\Arrtri}[3]{
	\left(
		\begin{array}{c}
		{#1} 
		\\
		{#2}
		\\
		{#3}		
		\end{array}
	\right)
}

\newcommand{\MATT}[4]{
	\left[
		\begin{array}{cc}
		{#1} 
		&
		{#2}
		\\
		{#3} 
		&
		{#4}		
		\end{array}
	\right]
}
\newcommand{\MATnine}[9]{
	\left[
		\begin{array}{ccc}
		{#1} & {#2} & {#3}
		\\
		{#4} & {#5} & {#6}
		\\
		{#7} & {#8} & {#9}		
		\end{array}
	\right]
}

\newcommand{\tA}{\tilde{A}}
\newcommand{\tB}{\tilde{B}}

\newcommand{\half}{\frac{1}{2}}

\newcommand{\average}[1]{\ensuremath{\lbrace\!\!\lbrace#1\rbrace\!\!\rbrace} } 
\newcommand{\jump}[1]{\ensuremath{[\![#1]\!]} }

\newcommand{\Th}{\mathcal{T}_h} 
\newcommand{\Ti}[1]{\mathcal{T}_{#1}} 

\newcommand{\SIG}{{\boldsymbol \sigma} }

\newcommand{\NOR}{{\boldsymbol n} }
\newcommand{\fB}{{\boldsymbol f} }
\newcommand{\vB}{{\boldsymbol v} }
\newcommand{\uB}{{\boldsymbol u} }
\newcommand{\wB}{{\boldsymbol w} }
\newcommand{\BO}[1]{{\boldsymbol #1} }
\newcommand{\lB}{{\boldsymbol \lambda} }
\newcommand{\phiB}{{\boldsymbol \varphi} }

\newcommand{\Bopt}{\mathcal{B}}
\newcommand{\Hopt}{\mathcal{H}}
\newcommand{\Zopt}{{\boldsymbol \theta}}
\newcommand{\EPS}{ \mathcal{E} }

\newcommand{\GAM}{ \Gamma }

\newcommand{\dotV}[2]{\left( {#1} , {#2} \right) }
\newcommand{\dotS}[2]{\left\langle #1 , #2 \right\rangle }

\newcommand{\DGnorm}[2]{\Vert {#1} \Vert_{\textrm{\tt DG}{#2}} }
\newcommand{\Bnorm}[2]{\Vert {#1} \Vert_{{\tt HDG} #2} }
\newcommand{\Lnorm}[1]{{ \Vert #1 \Vert }}

\newcommand{\frL}{\boldsymbol \ell}

\newtheorem{algorithm}{Algorithm}
\newtheorem{remark}{\it Remark}

\title{Analysis of Schwarz methods for a hybridizable discontinuous Galerkin discretization}

\author{Martin J.~Gander\and Soheil Hajian}

\begin{document}
\maketitle
\begin{abstract}
Schwarz methods are attractive parallel solvers for large scale linear
systems obtained when partial differential equations are
discretized. For hybridizable discontinuous Galerkin (HDG) methods,
this is a relatively new field of research, because HDG methods impose
continuity across elements using a Robin condition, while classical
Schwarz solvers use Dirichlet transmission conditions. Robin
conditions are used in optimized Schwarz methods to get faster
convergence compared to classical Schwarz methods, and this even
without overlap, when the Robin parameter is well chosen. We present
in this paper a rigorous convergence analysis of Schwarz methods for
the concrete case of hybridizable interior penalty (IPH) method. We
show that the penalization parameter needed for convergence of IPH
leads to slow convergence of the classical additive Schwarz method,
and propose a modified solver which leads to much faster
convergence. Our analysis is entirely at the discrete level, and thus
holds for arbitrary interfaces between two subdomains. We then
generalize the method to the case of many subdomains, including cross
points, and obtain a new class of preconditioners for Krylov subspace
methods which exhibit better convergence properties than the classical
additive Schwarz preconditioner.  We illustrate our results with
numerical experiments.
\end{abstract}
\pagestyle{myheadings}
\thispagestyle{plain}
\markboth{Martin J.~Gander and Soheil Hajian}{OSM and DG}

\section{Introduction}
We consider the elliptic model problem
\begin{equation}\label{eq:pde}
  \begin{array}{rcll}
    \eta(x) u(x) - \nabla \cdot ( a(x) \nabla u ) &=& 
    f,\quad & \textrm{in $\OM \subset \RE^2$},\\
    u &=& 0, & \textrm{on $\partial \OM$},
  \end{array}
\end{equation}
in the weak sense where $f \in \Lsp^2(\OM)$, $a(x) \in L^\infty(\OM)$
and uniformly positive, $\eta_0 \geq \eta(x) \geq 0$ and $\OM$ is
assumed to be a convex polygon for simplicity.  Any discretization of
this problem, for example by a finite element method (FEM) or a
discontinuous Galerkin (DG) method, leads to a large sparse linear
system
\begin{equation}
\label{eq:linsys}
	A \uB = \fB,
\end{equation}
where $\uB$ is the vector of degrees of freedom representing an
approximation of $u$ and $A$ represents the disretized differential
operator. In this paper we consider a hybridizable interior penalty
(IPH\footnote{We use the acronym IPH for {\it hybridizable interior
    penalty} because this has become the common abbreviation following
  its introduction in \cite{cockburn} as a member of the family of HDG
  methods.}) discretization which results in a symmetric positive
definite (s.p.d.)~matrix $A$.  An IPH discretization seeks $u_h \in
\Lsp^2(\OM)$ over a triangulation of the domain where $u_h$ is not
necessarily continuous across elements. As common to DG methods, IPH
imposes the continuity of the solution approximately through
penalization techniques, i.e.~penalizing jumps of $u_h$ across
elements in the bilinear form. The penalization is controlled by a
penalty parameter $\mu$.

Since the matrix $A$ of IPH is s.p.d.~and sparse, one can use the
Conjugate Gradient (CG) method to solve the linear system
(\ref{eq:linsys}). The convergence of CG slows down as the condition
number $\kappa(A)$ grows. It is not hard to show that $\kappa(A) =
O(h^{-2})$, where $h$ is the maximum diameter of the elements in the
triangulation, see for instance \cite{castillo}. Therefore
preconditioning is unavoidable and domain decomposition (DD)
preconditioners have been developed and studied for such
discretizations, see \cite{paola2011, karakashian}.  IPH as local
solvers were also used to precondition classical IP discretizations
\cite{blanca2}. One can also design a substructuring preconditioner
for a $p$-version of IPH with poly-logarithmic growth in the condition
number, see for details \cite{joachim1}. For a similar discretization
where the approximation is continuous inside subdomains but
discontinuous across subdomains, a substructuring preconditioner was
proposed and analyzed for the $h$-version with logarithmic growth in
the condition number, see \cite{dryja}.

A favorite preconditioner is the additive Schwarz preconditioner, for which
the set of unknowns is partitioned into overlapping or non-overlapping
subsets, corresponding to subdomains with maximum diameter $H$. In
this paper we only consider the non-overlapping case\footnote{There is
  a subtle difference between overlap at the continuous level of the
  subdomains, and the discrete level of unknowns, see
  \cite{gander2008schwarz}: no overlap at the level of unknowns means
  minimal overlap of one mesh size at the continuous level for
  classical discretizations like finite elements or finite
  differences. This becomes however even more subtle here with DG
  discretizations, since the discrete unknowns are coupled through
  Robin conditions, and no overlap at the level of unknowns really
  means no overlap at the continuous level, see
  \cite{hajian2013block}.} and for simplicity study first only two
subdomains, a generalization is given in Section \ref{sec:multi}.  The
non-overlapping two subdomain decomposition results in a natural
partitioning of the unknowns $\uB = (\uB_1, \uB_2)^\top$. The solution
of the linear system by the additive Schwarz method without overlap is
equivalent to the block Jacobi iteration
\begin{equation}
\label{eq:blockJacobi}
  M \mathbf{u}^{(n+1)} = N \mathbf{u}^{(n)} + \mathbf{f},\quad
   M = \MATT{A_{1}}{}{}{A_{2}},\ N = M - A.
\end{equation}
The matrix $M$ is also s.p.d.~and can be considered as a
preconditioner for CG. It can be shown that in this case we have
$\kappa(M^{-1} A) \leq O(h^{-1})$ in the absence of a coarse solver;
see \cite{karakashian}. Preconditioned CG satisfies then the
convergence factor estimate $\rho \leq \frac{\sqrt{\kappa(M^{-1}
    A)}-1}{\sqrt{\kappa(M^{-1} A)}+1}= 1 - O(\sqrt{h})$. 

On the other hand it has been recently shown in \cite{hajian2013block}
that the block Jacobi iteration in (\ref{eq:blockJacobi}) for an IPH
discretization can be viewed as a discretization of a non-overlapping
Schwarz method with Robin transmission conditions, i.e.~
\begin{equation}\label{eq:DD-DGH}
  \begin{array}{rcllrcll}
     (\eta-\Delta) u_1^{(n+1)} &=& f  & \text{in $\OM_1$},&
     (\eta-\Delta) u_2^{(n+1)} &=& f  & \text{in $\OM_2$},\\
     \Bopt_1 u_1^{(n+1)} &=& \Bopt_1 u_2^{(n)} & \text{on $\GAM$},&
     \Bopt_2 u_2^{(n+1)} &=& \Bopt_2 u_1^{(n)}& \text{on $\GAM$},
  \end{array}
\end{equation}
where $\Bopt_i w = \mu\, w + \PDif{w}{\NOR_i} $, $\GAM$ is the
interface between the two subdomains and $\mu$ is precisely the
penalty parameter of the IPH discretization.  This parameter $\mu$ has
to be chosen such that it ensures coercivity and optimal approximation
properties. For an IPH discretization, we must have
$\mu={\alpha}{h^{-1}}$ for some constant $\alpha>0$ large enough,
independent of $h$, and this scaling cannot be weakened, since
otherwise coercivity is lost. On the other hand, optimized Schwarz
theory suggests that the iteration in (\ref{eq:DD-DGH}) converges
faster if $\mu=O({h^{-1/2}})$, see \cite{ganderos}. In that case for
the contraction factor we have $\rho = 1 - O(\sqrt{h})$ while with the
choice $\mu = O(h^{-1})$ for IPH, we have $\rho = 1 - O(h)$.

The challenge is therefore to design a Schwarz algorithm for IPH 
with convergence factor  $\rho = 1 - O(\sqrt{h})$, while having the same
fixed point as the original additive Schwarz or block Jacobi method
for IPH. An idea for doing this can be found for Maxwell's equation in
\cite{dolean}. This approach was also adopted for IPH in
\cite{hajian2014}, where numerical experiments show that the
convergence factor is indeed $\rho = 1 - O(\sqrt{h})$, while
maintaining the same fixed point, but there is no convergence
analysis.

We provide in this paper a convergence theory for Schwarz methods
applied to IPH discretizations and prove these numerical
observations. A similar analysis exists for classical FEM using Schur
complement formulations and exploiting eigenvalues of the
Dirichlet-to-Neumann (DtN) operator, see \cite{lui}.  Our analysis
uses similar DtN arguments, but is substantially different from
\cite{lui}, since in a DG method continuity conditions are imposed
only weakly. We focus in our analysis on the $h$-version with
polynomial degree one, and do not study the effect of possible jumps
in $a(x)$ or higher polynomial degree.

Our paper is organized as follows: in Section \ref{sec:iph} we
describe two different but equivalent formulations of IPH, and
construct a Schur complement system. In Section \ref{sec:tech} we
provide mathematical tools to analyze Schwarz methods formulated using
Schur complements. In Section \ref{sec:schwarz} we present the
additive Schwarz and a new Schwarz algorithm for IPH in a two
subdomain setting and prove their convergence with concrete
contraction factor estimates. Section \ref{sec:multi} contains a
generalization of the algorithms to the multi-subdomain case. We show
in Section \ref{sec:num} numerical experiments to illustrate
our analysis, and also verify numerically that the new algorithm
provides a better preconditioner for Krylov subspace methods: we
observe that the contraction factor is $\rho = 1 - O(h^{1/4})$ which
is much faster than the CG solver preconditioned by one level additive
Schwarz.

\section{Hybridizable Interior Penalty method} \label{sec:iph}

This section is devoted to recall the definition of IPH in two
different but equivalent forms, namely the primal and hybridizable
formulation.  We later in Section \ref{sec:schwarz} design and analyze
two Schwarz methods for the hybridizable form and show that the first
one is slow and equivalent to a block Jacobi method applied to a
primal form, i.e.~(\ref{eq:blockJacobi}). However the second Schwarz
method takes advantage of hybridizable formulation and achieve faster
convergence.

IPH was first introduced in \cite{ewing} as a stabilized discontinuous
finite element method and later was studied as a member of the class
of hybridizable DG methods in \cite{cockburn}. It has been shown that
it is equivalent to a method called Ultra Weak Variational Formulation
(UWVF) for the Helmholtz equation; see \cite{MZA:8194617}. IPH also
fits into the framework developed in \cite{dgunified} for a unified
analysis of DG methods.  IPH is further studied in
\cite{lehrenfeld2010hybrid} in the context of incompressible flows.

\subsection{Notation} \label{sec:notation}
We follow the notation introduced in \cite{dgunified}.  Let
$\Th=\{K\}$ be a shape-regular and quasi-uniform triangulation of the
domain $\OM$.  Let $h_K$ be the diameter of an element of the
triangulation defined by $h_K := \max_{x,y \in K}|x-y|$ and $h =
\max_{K \in \Th} h_K$.  If $e$ is an edge of an element, we denote by
$h_e$ the length of that edge. The quasi-uniformity of the mesh
implies $h \approx h_K \approx h_e$.

We denote by $\EPS^0$ the set of interior edges shared by two elements
in $\Th$, that is
\begin{equation*}
\EPS^0 := \left\lbrace e = \partial K_1 \cap \partial
K_2, \forall K_1, K_2 \in \Th \right\rbrace,
\end{equation*}
by $\EPS^\partial $ the set of boundary edges, and all edges by $\EPS
:= \EPS^\partial \cup \EPS^0$.
We introduce the broken Sobolev space $\Hsp^l (\Th) := \prod_{K \in
  \Th} \Hsp^l(K)$
where $\Hsp^l(K)$ is the Sobolev space in $K \in \Th$ and $l$ is a
positive integer.  Note that $q \in \Hsp^l(\Th)$ is not necessarily
continuous across elements. Therefore the element boundary traces of
functions in $\Hsp^l (\Th)$ belong to $ \Tsp(\EPS) = \prod_{K \in \Th}
\Lsp^2( \partial K ) $, where $q \in \Tsp(\EPS)$ can be double-valued
on $\EPS^0$, but is single-valued on $\EPS^\partial$.

We now define two trace operators: let $q \in \Tsp(\EPS)$ and $q_i :=
\left. q \right|_{\partial K_i}$. Then on $e = \partial K_1 \cap
\partial K_2$ we define the average and jump operators
\begin{equation*}
  \begin{array}{lrlr}
    \average{q} := \frac{1}{2} ( q_1 + q_2 ), 
    & 
    &
    \jump{q} := q_1 \, \NOR_1 + q_2 \, \NOR_2,
    & 
  \end{array}
\end{equation*}
where $\NOR_i$ is the unit outward normal from $K_i$ on $e \in \EPS^0$.  It
is clear that these operators are independent of the element
enumeration.  Similarly for a vector-valued function $\SIG \in \left[
  \Tsp(\EPS) \right]^2 $ we define on interior edges
\begin{equation*}
  \begin{array}{lrlr}
    \average{\SIG} := \frac{1}{2} ( \SIG_1 + \SIG_2 ), 
    &
    &
    \jump{\SIG} := \SIG_1 \cdot \NOR_1 + \SIG_2 \cdot \NOR_2.
    &
  \end{array}
\end{equation*}
On the boundary, we set the average and jump operators to
$\average{\SIG} := \SIG$ and $\jump{q} = q \, \NOR$.
We do not need to define $\average{q}$ and $\jump{\SIG}$ on $e \in
\EPS^\partial$.

We define a finite dimensional subspace of $\Hsp^l(\Th)$ by
\begin{equation}
  V_h := 
  \left\lbrace v \in \Lsp^2(\OM) :
  \left. v \right|_{K} \in \PO^k(K), \forall K \in \Th \right\rbrace,
\end{equation}
where $\PO^k(K)$ is the space of polynomials of degree $\leq k$ in the
simplex $K \in \Th$.  We denote boundary integrals on an edge $e \in
\EPS$ by
\begin{equation*}
  \dotS{a}{b}_{e} := \int_{e} a \, b  \quad \textrm{if } a,b \in \Tsp(e),
  \quad
  \dotS{\BO{a}}{\BO{b}}_{e} := \int_{e} \BO{a} \cdot \BO{b} \quad
  \textrm{if } \BO{a},\BO{b} \in [\Tsp(e)]^2,
\end{equation*}
and similarly for volume terms on an element $K \in \Th$
\begin{equation*}
      \dotV{a}{b}_{K} := \int_{K} a \, b \quad \textrm{if } a,b \in
      \Hsp^l(K), \quad
      \dotV{\BO{a}}{\BO{b}}_{K} := \int_{K} \BO{a}
      \cdot \BO{b} \quad \textrm{if } \BO{a},\BO{b} \in [\Hsp^l(K)]^2.
\end{equation*}
If $\GAM$ is a subset of $\EPS$, we denote the $\Lsp^2$-norm of $q \in
\Tsp(\EPS)$ along $\GAM$ by $ \norm{q}_{\GAM}^{2} := \sum_{e \in \GAM}
\norm{q}^2_{e} $ and $\norm{q}^2_e := \dotS{q}{q}_{e}$.  Similarly if
$\Ti{i}$ is a subset of $\Th$, we denote the $\Lsp^2$-norm of a $v \in
\Hsp^l(\Ti{i})$ by $ \norm{v}_{\Ti{i}}^{2} := \sum_{K \in \Ti{i}}
\norm{v}_{K}^2 $.  

For $v \in \Hsp^1(\Th)$ we define functions whose restrictions to each
element, $K \in \Th$, are equal to the gradient of $v$. This operator
in the literature is called piecewise gradient and is usually denoted
by $\nabla_h$. For the sake of simplicity we use $\nabla v$ instead of
$\nabla_h v$.

\subsection{Primal formulation} 
To simplify our presentation, we set $\eta \geq 0 $ to be a constant
and $a(x)=1$ in the model problem (\ref{eq:pde}).  Let $u,v \in
\Hsp^{2}(\Th)$, then the IPH bilinear form of the model problem
(\ref{eq:pde}) is defined as
\begin{equation}
  \begin{array}{rcl}
    a(u,v) &:=& \eta \dotV{u}{v}_{\Th} + \dotV{\nabla u}{\nabla v}_{\Th}
    - \dotS{\average{\nabla u}}{\jump{v}}_{\EPS}
    - \dotS{\average{\nabla v}}{\jump{u}}_{\EPS}
    \\
    && + \dotS{\frac{\mu}{2} \jump{u}}{\jump{v}}_{\EPS} - 
    \dotS{\frac{1}{2\mu} \jump{\nabla u}}{\jump{\nabla v}}_{\EPS^0},
  \end{array}
\end{equation}
where $\mu \in \Tsp(\EPS)$, $\left. \mu \right|_{e} =
{\alpha}{h_e^{-1}}$ and $\alpha>0$.  Observe that $a(\cdot,\cdot)$ is
symmetric.  The definition of the IPH bilinear form is different from
the classical Interior Penalty (IP) method only in the last term,
i.e.~the last term in $a(\cdot,\cdot)$ is not present in IP.

There are two natural energy norms which are equivalent at the
discrete level.  Let $u \in V(h) := V_h + \Hsp^2(\OM) \cap
\Hsp^1_0(\OM) \subset \Hsp^2(\Th)$ then
\begin{equation}
  \begin{array}{lcl}
    \DGnorm{u}{}^{2} &:=& \eta \norm{u}_{\Th}^{2} +
    \norm{\nabla u}_{\Th}^{2} + 
    \sum_{e \in \EPS} \mu_e \norm{ \jump{u} }_{e}^{2},
    \\
    \DGnorm{u}{,\ast}^{2} &:=& 
    \DGnorm{u}{}^{2} + \sum_{K \in \Th} h_K^2 |u|_{K,2}^2.
  \end{array}
\end{equation}
One can show that they are equivalent at the discrete level by a local
application of the inverse inequality (\ref{eq:invineq}).
\begin{proposition}
Let $u \in V_h$. Then we have
\begin{equation*}
  \DGnorm{u}{}^{2} \leq \DGnorm{u}{,\ast}^{2} \leq {C}^2 \DGnorm{u}{}^{2},
\end{equation*}
where ${C}^2 > 1 $ and independent of $h$ and $\alpha$.
\end{proposition}
%
%

The norm $\DGnorm{\cdot}{,\ast}$ provides a natural norm for
boundedness and $\DGnorm{\cdot}{}$ can be used for showing
coercivity. The main ingredients for coercivity are the following
inequalities which hold for all $u \in V_h$:
\begin{equation} 
  \label{eq:coerc1} 
  \begin{array}{rcl}
    2 \dotS{\average{\nabla u} }{ \jump{u} }_{\EPS} &\leq& 
    \frac{1}{2} \norm{ \nabla u }^{2}_{\Th}	 
    + 
    \sum_{e \in \EPS } \frac{C_1}{h_e} \norm{ \jump{u} }_{e}^{2},
    \\
    \dotS{\frac{1}{2\mu} \jump{\nabla u}}{\jump{\nabla u}}_{\EPS^0} &\leq&
    \frac{C_2}{\alpha} \norm{ \nabla u }^2_{\Th},
  \end{array}
\end{equation}
where $C_1$ and $C_2$ are both independent of $h$ and $\alpha$ but
depend on the polynomial degree.  This can be obtained from the trace
inequality
\begin{equation} 
  \label{eq:warburton}
  \norm{w}_{\partial K}^2 \leq c \frac{k^2}{h} \norm{w}_K^2, \quad
  \forall w \in \PO^k(K),
\end{equation}
where $k$ is the polynomial degree, for details see
\cite{hesthaven, dgunified}.
\begin{proposition}
If $\mu = {\alpha}{h^{-1}}$, for $\alpha>0$ and sufficiently large,
then we have
\begin{equation*}
  \begin{array}{rcccll}
    && a(u,v) &\leq& 
    \overline{C} \DGnorm{u}{,\ast} \DGnorm{v}{,\ast} & \forall u,v \in V(h),
    \\
    \underline{c} \, C^{-2} \DGnorm{u}{,\ast}^2 \leq
    \underline{c} \DGnorm{u}{}^2 & \leq & a(u,u) & & & \forall u \in V_h,
  \end{array}
\end{equation*}
where $\underline{c} = \min\{ \frac{1}{2} - \frac{C_2}{\alpha} , 1 -
\frac{C_1}{\alpha} \} < 1$ , $\overline{C} = 1 + \frac{C_3}{\alpha} >
1$ and both constants are independent of $h$.
\end{proposition}

Note that coercivity holds only for $u \in V_h$ and that
$\alpha>0$ has to be big enough to result in a positive
$\underline{c}$. Since $C_1$ and $C_2$ come from the trace inequality,
we can choose $\alpha = O(k^2)$ where $k$ is the degree of the
polynomials in the simplex.  Throughout this
paper we assume that $\alpha$ is chosen big enough to ensure that any
term of type $1 - \frac{c}{\alpha}$ (with $c>0$, independent of $h$
and $\alpha$) is positive.

Having established that $a(\cdot,\cdot)$ is bounded and coercive, we
obtain that the following approximation problem has a unique
solution: find $u_h \in V_h$ such that
\begin{equation} \label{eq:varIPH}
  a(u_h,v) = \dotV{f}{v}_{\Th}, \quad \forall v \in V_h.
\end{equation}
Assuming the exact solution is regular enough, it can
be shown that
\begin{equation*}
  \begin{array}{lcl}
    \DGnorm{u_h - u}{,\ast} &\leq& c\, h^{k} | u |_{k+1,\OM},
    \\
    \norm{u_h - u}_{0} &\leq& c\, h^{k+1} | u |_{k+1,\OM},
  \end{array}
\end{equation*}
i.e.~IPH has optimal approximation order
\cite{dgunified,lehrenfeld2010hybrid}. We emphasize that
without setting $\mu = \alpha h^{-1}$, the coercivity and optimal
approximation properties are lost.

\subsection{Hybridizable formulation}

In this section we exploit the fact that IPH is a hybridizable
method. A method is hybridizable if one can eliminate the degrees of
freedom inside each element to obtain a linear system in terms of a
single-valued function along the edges, say $\lambda_h$.  Not all DG
methods have this property, for example classical IP is not
hybridizable. A unified hybridization procedure for DG methods has
been introduced and studied in \cite{cockburn} where IPH is also
included.

We introduce the general setting by decomposing the domain into
two non-overlapping subdomains $\OM_1$ and $\OM_2$. Denoting the
interface by $\GAM := \overline{\OM}_1 \cap \overline{\OM}_2$, we
assume $\GAM \subset \EPS^0$, i.e.~the cut does not go through any
element of the triangulation. This will result in a natural
partitioning of $\Th$ into $\Ti{1}$ and $\Ti{2}$ which do not overlap
but share $\GAM$ as a boundary; see for an example Figure
\ref{fig:ddmesh}.  
\begin{figure}
  \centering
  \epsfig{file=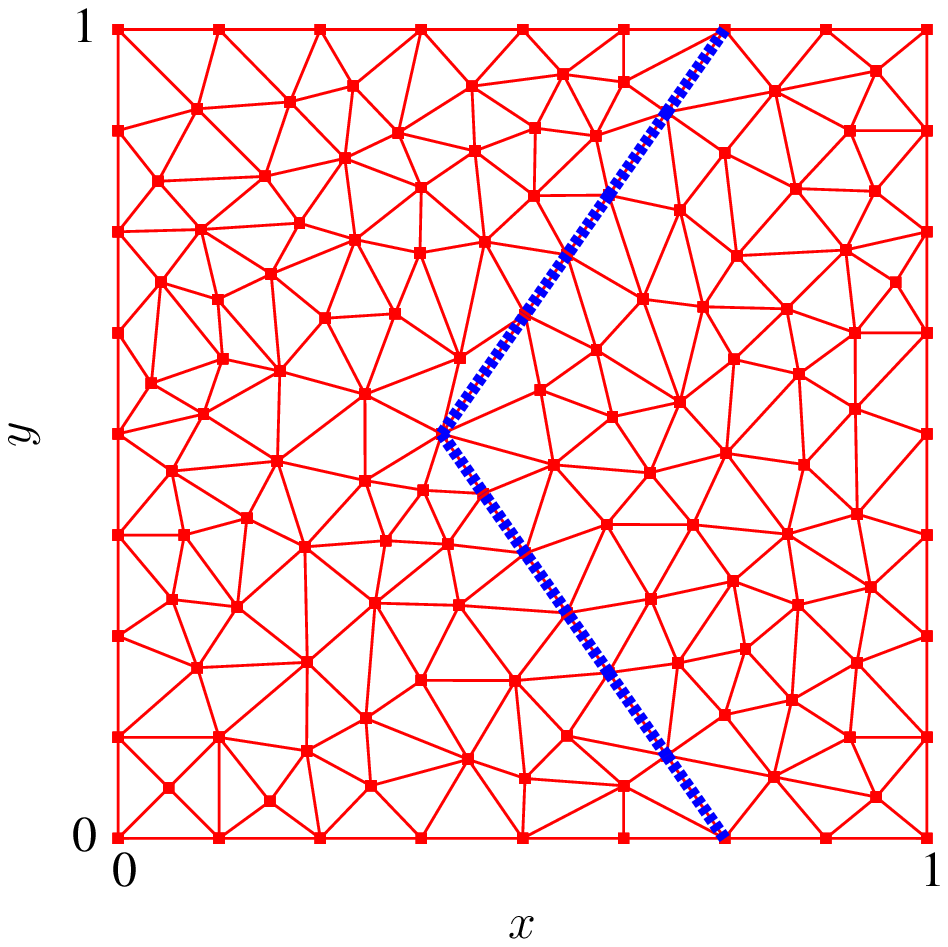, scale=0.6}
  \hspace{1cm}
  \epsfig{file=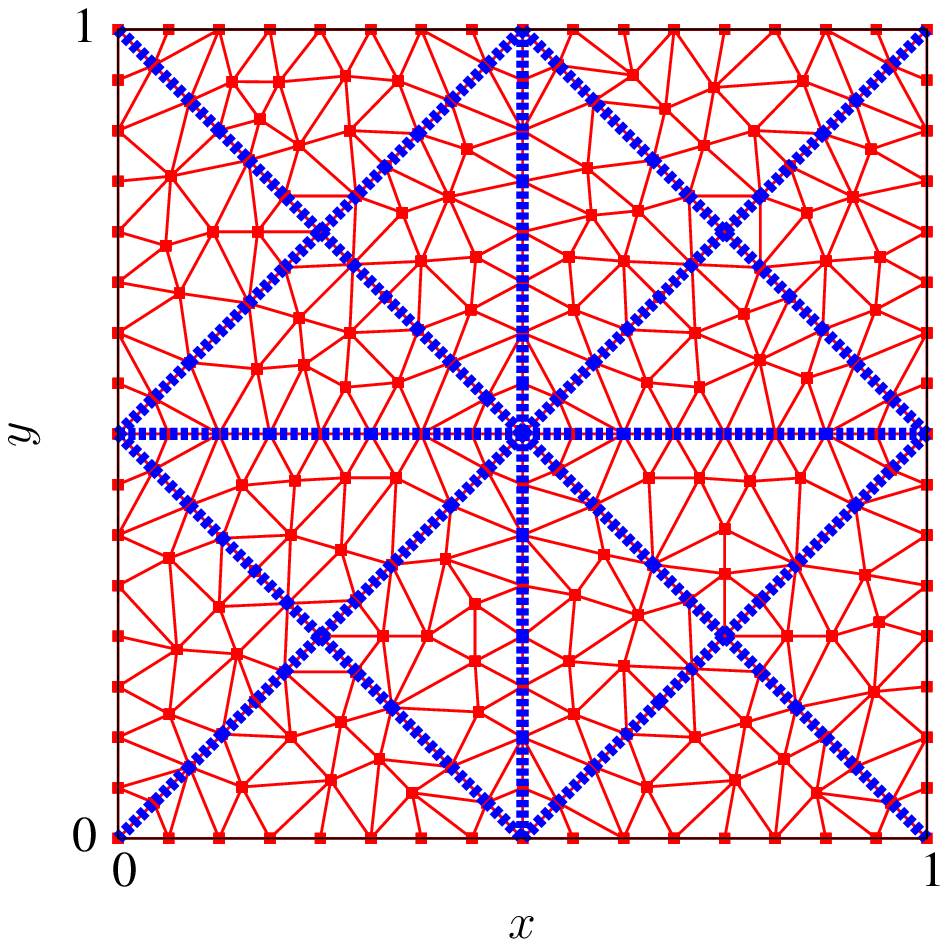, scale=0.6}	
  \caption{An unstructured mesh with the interface $\Gamma$
    (thick-dashed).}
  \label{fig:ddmesh}
\end{figure}
We denote by $H$ the maximum diameter of the subdomains and by $H_\OM$
the diameter of the mono-domain $\Omega$. We assume $0 < h
\leq H < H_\OM$.

We introduce local spaces on $\OM_1$ and $\OM_2$ by
\begin{equation}
  V_{h,i} := \big\{ v \in \Lsp^2(\OM_i) : \left. v \right|_{K \in \Ti{i}} \in \PO^k(K) \big\}, 
  \text{ for } i=1,2.
\end{equation} 
Note that this domain decomposition setting implies $V_h =
V_{h,1} \oplus V_{h,2}$.  We define on the interface the space of
broken single-valued functions by
\begin{equation}
  \Lambda_h := 
  \big\{ \varphi \in \Lsp^2(\GAM) : \left. \varphi \right|_{e \in \GAM} \in \PO^k(e) \big\}.
\end{equation}
For the sake of simplicity we denote the restriction of $v \in V_{h}$ on
$V_{h,i}$ by $v_i$.  Observe that the trace of $v_i \in V_{h,i}$ on
$\GAM$ belongs to $\Lambda_h$.

Let $(u,\lambda), (v,\varphi) \in V_{h} \times \Lambda_h$ and consider
the symmetric bilinear form
\begin{equation}
  \tilde{a}( (u,\lambda), (v,\varphi) ) := \tilde{a}_{\GAM}(\lambda,\varphi) + 
  \sum_{i=1}^{2} \Big( \tilde{a}_i(u_i,v_i) 
  + \tilde{a}_{i\GAM}(v_i,\lambda) + \tilde{a}_{i\GAM}(u_i,\varphi) \Big),
\end{equation}
where
\begin{equation}
  \label{eq:deftildea}
  \begin{array}{rcl}
    \tilde{a}_{\GAM}(\lambda,\varphi) &:=& 2 \dotS{\mu \, \lambda}{\varphi}_{\GAM},
    \\
    \tilde{a}_{i\GAM}(v_i,\varphi) &:=& 
    \dotS{ \PDif{v_i}{\NOR_i}  - \mu v_i }{\varphi}_{\GAM},
  \end{array}
\end{equation}
and
\begin{equation}
  \label{eq:IPHloc}
  \begin{array}{rcl}
    \tilde{a}_i(u_i,v_i) &:=& 
    \eta \dotV{u_i}{v_i}_{\Ti{i}} + \dotV{\nabla u_i}{\nabla v_i}_{\Ti{i}}
    - \dotS{\average{\nabla u_i}}{\jump{v_i}}_{\EPS_i^0} 
    - \dotS{\average{\nabla v_i}}{\jump{u_i}}_{\EPS_i^0}
    \\
    && + \dotS{\frac{\mu}{2} \jump{u_i}}{\jump{v_i}}_{\EPS_i^0}
    - \dotS{\frac{1}{2\mu} \jump{\nabla u_i}}{\jump{\nabla v_i}}_{\EPS_i^0}
    \\
    && - \dotS{ \PDif{u_i}{\NOR_i}}{v_i}_{\partial \OM_i}
    - \dotS{ \PDif{v_i}{\NOR_i}}{u_i}_{\partial \OM_i}
    + \dotS{ {\mu} \, {u_i}}{{v_i}}_{\partial \OM_i}.
  \end{array}
\end{equation}

This is an IPH discretization of the model problem in $\OM_i$ and $\partial
\OM_i$ is treated as a Dirichlet boundary. Therefore
$\tilde{a}_i(\cdot,\cdot)$ inherits coercivity and continuity of the
original bilinear form, $a(\cdot,\cdot)$.

The global bilinear form $\tilde{a}(\cdot,\cdot)$ is also coercive at
the discrete level, if $\alpha>0$ is sufficiently large, independent of
$h$. To see this we introduce an energy norm for all $(v_i,\varphi)
\in V_{h,i} \times \Lambda_h$ such that
\begin{equation}
  \Bnorm{(v_i,\varphi)}{,i}^{2} := 
  \eta \Lnorm{ v_i }_{\Ti{i}}^{2} +
  \Lnorm{ \nabla v_i }_{\Ti{i}}^{2} + 
  \mu \Lnorm{ \jump{v_i} }_{\EPS_i \setminus \GAM}^{2}
  + \mu \Lnorm{ v_i - \varphi }_{\GAM}^{2}, \quad (i=1,2).
\end{equation}
then by definition of $\tilde{a}(\cdot,\cdot)$ for all $(v,\varphi)
\in V_h \times \Lambda_h$ we have
\begin{equation} \label{eq:coercAtilde}
  \begin{array}{rcl}
    \tilde{a}( (v,\varphi), (v,\varphi) ) &=&
    \tilde{a}_\GAM(\varphi,\varphi) + \sum_{i=1}^2
    \big( \tilde{a}_i(v_i,v_i) + 2 \tilde{a}_{i\GAM}(v_i,\varphi) \big),
    \\ 
    &=& \sum_{i=1}^2 \big( \tilde{a}_i(v_i,v_i) + 2 \tilde{a}_{i\GAM}(v_i,\varphi)
    + \half \tilde{a}_\GAM(\varphi,\varphi) \big).
  \end{array}
\end{equation}
We can bound the contribution of each subdomain from below
separately: 
\begin{equation*}
  \begin{array}{rcl}
    \tilde{a}( (v,\varphi), (v,\varphi) ) &=& 
    \sum_{i=1}^2 \eta \norm{v_i}_{\Ti{i}}^2 +
    \norm{\nabla v_i}_{\Ti{i}}^2 
    \\
    && \quad  
    - 2 \dotS{\average{\nabla v_i} }{ \jump{v_i} }_{\EPS_i \setminus \GAM}	
    + \frac{\mu}{2} \norm{\jump{v_i}}_{\EPS_i \setminus \GAM}^2
    - \frac{1}{2 \mu} \norm{ \jump{\nabla v_i} }_{\EPS_i^0}^2
    \\
    && \quad
    -2 \dotS{\PDif{v_i}{\NOR_i}}{v_i - \varphi}_\GAM 
    + \mu \norm{ v_i - \varphi }_\GAM^2,
    \\
    &\geq& c \sum_{i=1}^2 \Bnorm{ (v_i, \varphi) }{,i}^2,
  \end{array}
\end{equation*}
where we used the inverse inequalities (\ref{eq:warburton}) for terms
acting on the interface and (\ref{eq:coerc1}) for terms acting inside
subdomains. Here $0<c<1$ is a constant independent of $h$. Note that
we proved the coercivity in a subdomain by subdomain fashion by
splitting the $\tilde{a}_\GAM(\cdot,\cdot)$ terms.

Consider the following discrete problem: find $(u_h,\lambda_h) \in V_h
\times \Lambda_h$ such that
\begin{equation}
  \label{eq:varIPH2}
  \tilde{a}( (u_h,\lambda_h), (v,\varphi) ) = 
  \dotV{f}{v}_{\Th}, \quad \forall (v,\varphi) \in V_h \times \Lambda_h,
\end{equation}
which has a unique solution since $\tilde{a}(\cdot,\cdot)$ is coercive
on $V_h \times \Lambda_h$.  One can eliminate the interface variable,
$\lambda_h$, and obtain a variational problem in terms of $u_h$
only. It turns out that this coincides with the
variational problem (\ref{eq:varIPH}); for a proof see
\cite{lehrenfeld2010hybrid}.

The advantage of the variational problem (\ref{eq:varIPH2}) is that
each subproblem is communicating through the auxiliary unknown
$\lambda_h$. Therefore we can eliminate the interior unknowns, $u_i$,
and obtain a Schur complement system. If we test (\ref{eq:varIPH2})
with $v_i\not=0$, $v_j=0$ $(j\not=i)$, $\varphi = 0$ and assume
that $\lambda_h$ is known, we obtain a local problem: find $u_{i} \in
V_{h,i}$ such that
\begin{equation} \label{eq:harmonicsat}
  \tilde{a}_i(u_i,v_i) + \tilde{a}_{i\GAM}(v_i,\lambda_h) =
  \dotV{f}{v_i}_{\Ti{i}}, \quad \forall v_i \in V_{h,i}.
\end{equation}
This is an IPH discretization of the continuous problem
\begin{equation*}
  \begin{array}{rcll}
    (\eta-\Delta) u &=& f,\quad & \text{in $\OM_i$},\\
    u &=& \lambda_h, \quad & \text{on $\GAM$}, \\
    u &=& 0, & \text{on $\partial \OM_i \setminus \Gamma$}.
  \end{array}
\end{equation*}
However the boundary condition on $\GAM$ is imposed weakly and therefore
$u_i |_{\GAM} \not = \lambda_h$ in the strong sense, see
\cite{cockburn,hajian2013block,lehrenfeld2010hybrid}.

\subsection{Schur complement formulation} \label{sec:schur}
We choose nodal basis functions for $\PO^k(K)$ and denote the space of
degrees of freedom (DOFs) of $V_h$ by $V$ and similarly for subspaces
by $\{ V_i \}$. The variational form in (\ref{eq:varIPH}) is
equivalent to the linear system $A \uB = \fB$. $A$ is the system
matrix and $\uB \in V$ are the corresponding DOFs of the approximation
$u_h \in V_h$. We can partition $\uB$ into $\{ \uB_i\}$ where $\uB_i$
corresponds to DOFs of $u_{i} \in V_{h,i}$. Then we can arrange the
entries of $A$ and rewrite the linear system as
\begin{equation}
  \MATT{A_1}{A_{12}}{A_{21}}{A_{2}} \Arr{\uB_1}{\uB_2} = \Arr{\fB_1}{\fB_2}.
  \label{eq:linsyspart}
\end{equation}
We use nodal basis functions for $\Lambda_h$ and denote by $\lB$ the
corresponding DOFs for $\lambda_h \in \Lambda_h$.  Then
the variational form (\ref{eq:varIPH2}) can be written as
\begin{equation} \label{eq:DDshur}
  \MATnine{\tA_1}{}{\tA_{1\GAM}}{}{\tA_2}{\tA_{2\GAM}}{\tA_{\GAM 1}}{\tA_{\GAM 2}}{\tA_{\GAM}}
  \Arrtri{\uB_1}{\uB_2}{\lB}
  = \Arrtri{\fB_1}{\fB_2}{0},
\end{equation}
where $\tA_{\GAM i}^{} = \tA_{i\GAM}^{\top}$.  Since this matrix is
s.p.d.~and the same holds also for its diagonal blocks, we can form a
Schur complement system. We define
$
\tB_i := \tA_{\GAM i}^{} \tA_{i}^{-1} \tA_{i\GAM}^{}
$
and
$
\BO{g}_{\GAM}^{} :=  - \sum_{i=1}^{2} \tA_{\GAM i}^{} \tA_{i}^{-1} \fB_i^{}.
$
Then the Schur complement system reads
\begin{equation} \label{eq:shur}
  \tilde{S}_\GAM \lB := \Big( \tA_{\GAM} - \sum_{i=1}^{2} \tB_i \Big) \lB = \BO{g}_{\GAM}.
\end{equation}
\begin{definition}[discrete harmonic extension]\label{HarmExtDef}
For all $\varphi \in \Lambda_h$, we denote by $\Hopt_i(\varphi) \in
V_{h,i}$ the discrete harmonic extension into $\OM_i$,
\begin{equation}
  \Hopt_i(\varphi) \equiv - \tA_{i}^{-1} \tA_{i\GAM}^{} \phiB.
\end{equation}
The corresponding $\varphi$ is called {\normalfont generator}.  In
other words $u_{i} := \Hopt_i(\varphi)$ is an approximation obtained
from the IPH discretization in $\OM_i$ using $\varphi$ as Dirichlet data;
i.e.~$ \tA_{i} \uB_i + \tA_{i\GAM} \phiB = 0 $.
\end{definition}

The following result shows that an application of $\tB_i
\lB$ can be viewed as finding the harmonic extension, $u_i :=
\Hopt_i(\lambda_h)$, and then evaluating a ``Robin-like trace'' on the
interface.
\begin{proposition} \label{prop:Bopt}
Let $\lambda_h \in \Lambda_h$ and define its harmonic extension by $u_i :=
\Hopt_i(\lambda_h)$.  Then
$
\phiB^{\top} \tB_i \lB = \dotS{ \mu u_i - \PDif{u_i}{\NOR_i} }{\varphi}_\GAM
$
for all $\varphi \in \Lambda_h$.
\end{proposition}
\begin{proof}
Let $u_i := \Hopt_i(\lambda_h)$. Then by definition of $\tB_i$ and
$\tilde{a}_{i\GAM}(\cdot,\cdot)$ we have\\
\begin{equation*}
  \phiB^{\top} \tB_i^{} \lB = \phiB^{\top} \tA_{\GAM i}^{} \tA_{i}^{-1} \tA_{i\GAM}^{} \lB 
  = - \phiB^{\top} \tA_{\GAM i}^{} \uB_i^{} = 
  \dotS{ \mu u_i - \PDif{u_i}{\NOR_i} }{\varphi}_\GAM,
\end{equation*}
for all $\varphi \in \Lambda_h$,
which completes the proof, since $\tA_{\GAM i}^{} =
\tA_{i\GAM}^{\top}$. \quad
\end{proof}

\section{Properties of the Schur complement and technical tools} 
\label{sec:tech}

The main goal of this section is to provide estimates for the minimum
and maximum eigenvalues of the $\tilde{S}_\GAM$ and $\tB_i$ for
$i=1,2$. We use the estimate for the $\tB_i$ operators to prove
convergence of the Schwarz method and provide the contraction factor
later in Section \ref{sec:schwarz}. In particular we prove in this
section that the following estimates hold for all $\varphi \in
\Lambda_h$:
\begin{eqnarray}
  \label{eq:Best}
  c_B \, \mu \norm{\varphi}_\GAM^2 &\leq \phiB^\top \tB_i \phiB 
  \leq& \Big( 1 - C_B \frac{h}{H \alpha} \Big) \mu \norm{\varphi}_\GAM^2,
  \\
  \label{eq:Sest}
  c \frac{H}{H_\OM^2} \norm{\varphi}_\GAM^2 &\leq \phiB^\top \tilde{S}_\GAM \phiB 
  \leq& C \frac{\alpha}{h} \norm{\varphi}_\GAM^2,
\end{eqnarray}
where all constants are positive and independent of $h$, $H$ and
$H_\OM$. Since $\tilde{S}_\GAM$ and $\tB_i$ are symmetric, we can use
Rayleigh quotient arguments and obtain an estimate for the minimum and
maximum eigenvalues. One can also obtain an estimate with polynomial
degree dependency using the techniques of this section. 

The only constraint on the shape of the subdomains is a star-shape
assumption.  To prove the above estimates we need trace and Poincar\'e
inequalities for totally discontinuous functions.  The following trace
estimate is due to Feng and Karakashian \cite[Lemma
  3.1]{karakashian}. The Poincar\'e inequality is due to Brenner, see
\cite{brenner-poincare}.
\begin{lemma}[Trace inequality] \label{lemma:karakashian}
  Let $D$ be a star-shape domain with diameter $H_D$, and
  triangulation $\Th$.  Then, for any $u \in {\normalfont
    \Hsp}^1(\Th)$, we have 
  \begin{equation*}
    \norm{ u }_{\partial D}^{2} \leq c 
    \Big[ H_D^{-1} \norm{u}_{D}^{2} + H_D^{} \big( \norm{ \nabla u }_{D}^{2} + 
      h^{-1} \norm{\jump{u}}_{\EPS \setminus \partial D}^{2} \big) \Big].
  \end{equation*}
\end{lemma}
\begin{lemma}[Poincar\'e inequality]\label{lemma:poincare}
Let $D$ be an open connected polygonal domain with diameter
$H_D$, and triangulation $\Th$. Then, for any $u \in
{\normalfont \Hsp}^1(\Th)$ we have
\begin{equation*}
  \norm{ u }_{D}^2 \leq c H_D^2 
  \Big[ \norm{ \nabla u }_{D}^2 
  + h^{-1} \norm{ \jump{ u } }_{\EPS \setminus \partial D}^2
  + h^{-1} \norm{ { u } }_{ \nu}^2
  \Big],
\end{equation*}
where $\nu$ is a measurable subset of $\partial D$ with nonzero measure. 
\end{lemma}

\subsection{Eigenvalue estimates for $\tB_i$}

In order to obtain estimates for the eigenvalues of the $\tB_i$
operator, we first recall Definition \ref{HarmExtDef} of a harmonic
extension: $u_i \in V_{h,i}$ is called harmonic extension of $\varphi
\in \Lambda_h$ if it satisfies $\tA_i \uB_i + \tA_{i\GAM} \phiB =
0$. Now multiplying this relation by $\uB_i^\top$ from left we get
\begin{equation*}
  \begin{array}{llcl}
    &
    \uB_i^\top \tA_i^{} \uB_i^{} + \uB_i^\top \tA_{i\GAM}^{} \phiB^{}
    &=& 0
    \\
    \Leftrightarrow &
    \uB_i^\top \tA_i^{} \uB_i^{} - \phiB^\top \tA_{\GAM i}^{}
    \tA_i^{-1} \tA_{i\GAM}^{} \phiB^{}
    &=& 0
    \\
    \Leftrightarrow &
    \uB_i^\top \tA_i^{} \uB_i^{} - \phiB^\top \tB_i \phiB &=& 0,
  \end{array}
\end{equation*}
where we used $\uB_i^{} = - \tA_i^{-1} \tA_{i\GAM}^{} \phiB$,
$\tA_{\GAM i}^{} = \tA_{i\GAM}^\top$ and the definition of
$\tB_i$. Hence if $u_i = \Hopt_i(\varphi)$ then we have
\begin{equation} \label{eq:BtoA}
  \phiB^\top \tB_i \phiB = \tilde{a}_i(u_i, u_i).
\end{equation}
Now recall that $\tilde{a}_i(\cdot,\cdot)$ is coercive and bounded
over $V_{h,i}$, therefore $\underline{c} \DGnorm{u_i}{}^2 \leq
\tilde{a}_i(u_i,u_i) \leq \overline{C} \DGnorm{u_i}{}^2$.  Thus if we
relate the energy norm of the harmonic extension, $u_i :=
\Hopt_i(\varphi) \in V_{h,i}$, to the $\Lsp^2$-norm of $\varphi$ we
obtain the desired estimate (\ref{eq:Best}). More precisely we can
show that the estimate
\begin{equation} \label{eq:harmIneq}
  {c}_{\Hopt} \cdot \mu \norm{ \varphi }_{\GAM}^{2} \leq 
  \DGnorm{ u_i }{}^{2} 
  \leq {C}_{\Hopt} \cdot \mu \norm{ \varphi }_{\GAM}^{2}
\end{equation}  
holds, where $0 < c_\Hopt < 1$ and $C_\Hopt > 1$ are constants
independent of $h$. Observe that $C_\Hopt > 1$ while the upper bound
estimate in (\ref{eq:Best}) is less than one. We show later how one can
obtain a sharp upper bound estimate as in (\ref{eq:Best}).

Let us start with the lower bound of inequality (\ref{eq:harmIneq}).
First we introduce an extension by zero operator $ \Zopt_i : \Lambda_h
\rightarrow V_{h,i} $ which is defined for all $\varphi \in \Lambda_h$
as
\begin{equation*}
  \Zopt_i( \varphi ) :=
  \left\{
  \begin{array}{ll}
    \varphi & \text{on edges belonging to $\GAM$},
    \\
    0	& \text{on other nodes}.
  \end{array}
  \right.
\end{equation*}
For a graphical illustration see Figure \ref{fig:extzero}.  
\begin{figure}
  \centering \epsfig{file=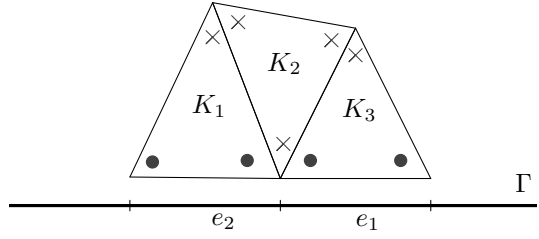, scale=1}
  \caption{Illustration of the extension by zero, $\Zopt_i(\varphi)$,
    for elements which share an edge with the interface,
    e.g.~$\{K_1,K_3\}$, and those which do not, e.g.~$K_2$.}
 \label{fig:extzero}
\end{figure}
Note that there are elements like $K_2$ which physically share a node
and not an {\it edge} with the interface, but we leave
$\Zopt_i(\varphi)$ in $K_2$ to be zero.  More precisely, only those
elements which share an edge with the interface are non-zero.

We show in the Appendix, see also \cite{qin}, that in an element, $K
\in \Ti{i}$, with an edge $e \in \GAM$ we have
\begin{equation} \label{eq:zoptineq}
  \begin{array}{lcl}
    \norm{ \Zopt_i (\varphi ) }_{K}^2
    &\leq& C_3\, h \norm{ \varphi }_{e}^2,
    \\
    \norm{ \nabla \Zopt_i( \varphi ) }_{K}^2 &\leq& {C_4}{h^{-1}}
    \norm{ \varphi }_{e}^2, 
    \\ \norm{ \jump{ \Zopt_i(\varphi ) } }_{\EPS_i}^2 &\leq& C_5 \norm{ \varphi
    }_{\GAM}^{2},
  \end{array}
\end{equation}
where $C_3>0$,  $C_4>0$ and $ C_5 \geq 1 $ and all are independent of
$h$. This yields the following result which relates the energy of the
extension by zero to its $\Lsp^2$-norm on the interface.
\begin{lemma} \label{lemma:Zoptphi}
  Let $\varphi \in \Lambda_h$ and $\Zopt_i( \varphi )$ be its extension
  by zero into $\OM_i$. We have
  \begin{equation*}
    \DGnorm{ \Zopt_i(\varphi) }{}^{2} \leq \mu\, C_\theta \, \norm{
      \varphi }_{\GAM}^{2},
  \end{equation*}
  where $C_\theta = C_3 \eta + C_4 \alpha^{-1} + C_5 > 1$.
\end{lemma}
\begin{proof}
  First note that by definition $\Zopt_i(\varphi)$ and $ \nabla
  \Zopt_i(\varphi) $ are non-zero only on those elements which share
  an edge with the interface. We call them $\{ K_{\GAM} \} \subset
  \Ti{i}$. Then we have
  \begin{equation*}
    \begin{array}{rcll}
      \DGnorm{ \Zopt_i(\varphi) }{}^2 &=& 
      \displaystyle 
      \sum_{K \in \{K_{\GAM} \}} \eta \norm{ \Zopt_i( \varphi ) }_{K}^2 +
      \norm{ \nabla \Zopt_i( \varphi ) }_{K}^2 + 
      \mu \norm{ \jump{ \Zopt_i(\varphi) } }_{\EPS_i}^{2}
      \\
      &\leq& 
      C_3 \, \eta \, h \norm{ \varphi }_{\GAM}^{2} +
      \frac{C_4}{{h}} \norm{ \varphi }_{\GAM}^{2} + 
      C_5 \, \mu \norm{ \varphi }_{\GAM}^{2}
      \\			
      &\leq&
      \mu \left( C_3 \, \eta + \frac{C_4}{\alpha}  + C_5  \right) \norm{ \varphi }_{\GAM}^{2},
    \end{array}
\end{equation*}
which completes the proof with $C_\theta :=  C_3 \, \eta + \frac{C_4}{\alpha} + C_5 
> 1$. \quad
\end{proof}

Now we are able to relate the energy of a harmonic extension, $u_i
:= \Hopt_i(\varphi)$, to the $\Lsp^2$-norm of $\varphi$ on the
interface. 
\begin{lemma} \label{lemma:trinvphi}
  Let $\varphi \in \Lambda_h$ and $u_i := \Hopt_i(\varphi)$ be its
  harmonic extension into $\OM_i$. Then we have
  \begin{equation*}
    {c}_{\Hopt} \cdot \mu \norm{ \varphi }_{\GAM}^{2} \leq \DGnorm{ u_i }{}^{2},
  \end{equation*}
  where $c_{\Hopt} = (1 - \frac{c}{\alpha})^2 \cdot \frac{1}{C_\theta
    \overline{C}^2} < 1$.
\end{lemma}
\begin{proof}
  Since $u_i$ is the harmonic extension of $\varphi$, it satisfies
  (\ref{eq:harmonicsat}) (with $f=0$).  Let $v =
  \Zopt_i(\varphi)$. Then by definition of
  $\tilde{a}_{i\GAM}(\cdot,\cdot)$ we have
  \begin{equation*}
    \tilde{a}_i(u_i,\Zopt_i(\varphi) ) = 
    \dotS{ \mu\, \Zopt_i(\varphi) - \PDif{\Zopt_i(\varphi)}{\NOR_i}}{ \varphi }_{\GAM}.
  \end{equation*}
  Note that $\Zopt_i(\varphi)|_{\GAM} = \varphi$. We can bound the
  right-hand side from below, therefore
  \begin{equation*}
    \begin{array}{rcll}
      \tilde{a}_i(u_i,\Zopt_i(\varphi) ) &\geq& \mu \norm{\varphi}_{\GAM}^{2} -
      \norm{ \PDif{\Zopt_i(\varphi)}{\NOR_i} }_{\GAM} \,\norm{ \varphi }_{\GAM}
      &
      \\
      &\geq&
      \mu \norm{\varphi}_{\GAM}^{2} -
      \frac{c}{\sqrt{h}}\norm{ \nabla {\Zopt_i(\varphi)} }_{K_\GAM}  
      \,\norm{ \varphi }_{ \GAM }
      &
      \quad \textrm{by ineq.~(\ref{eq:warburton})}
      \\		
      &\geq&
      \mu \norm{\varphi}_{\GAM}^{2} -
      \frac{c'}{h} \, \norm{ \varphi }_{ \GAM }^{2}
      & \quad \textrm{by ineq.~(\ref{eq:zoptineq})}
      \\	
      &=&
      \mu \left( 1 - \frac{c'}{\alpha} \right) \norm{ \varphi }_{ \GAM }^{2},
    \end{array}
  \end{equation*}
  which is positive if $\alpha>0$ and sufficiently large. By
  continuity of $\tilde{a}_{i}(\cdot,\cdot)$ we have
  \begin{equation*}
    \mu \left( 1 - \frac{c'}{\alpha} \right) \norm{ \varphi }_{ \GAM }^{2} 
    \leq \overline{C} \, \DGnorm{u_i}{} \cdot 
    \DGnorm{ \Zopt_i(\varphi)}{}.
  \end{equation*}
  Note that we are able to use $\DGnorm{ \cdot }{}$ instead of
  $\DGnorm{ \cdot }{,\ast}$ since we work with discrete spaces.  An
  application of Lemma \ref{lemma:Zoptphi} completes the proof with
  $c_{\Hopt} = ( 1 - \frac{c'}{\alpha} )^2 \cdot \frac{1}{C_\theta
    \overline{C}^2} < 1$.  \quad
\end{proof}

The upper bound in (\ref{eq:harmIneq}) can be obtained much easier
using coercivity of the $\tilde{a}_i(\cdot,\cdot)$.
\begin{lemma}
Let $\varphi \in \Lambda_h$ and $u_i := \Hopt_i(\varphi)$ be its harmonic
extension into $\OM_i$. Then we have
\begin{equation*}
  \DGnorm{ u_i }{}^{2} \leq {C}_{\Hopt} \cdot \mu \norm{ \varphi }_{\GAM}^{2},
\end{equation*}
where $C_{\Hopt} = \left( 1 + \frac{C}{\sqrt{\alpha} } \right)^{2}
\cdot \frac{1}{\underline{c}^2} > 1$.
\end{lemma}
\begin{proof}
  Since $u_i$ is the harmonic extension of $\varphi$, it satisfies
  (\ref{eq:harmonicsat}) (with $f=0$).  Using the fact that
  $\tilde{a}_i(\cdot,\cdot)$ is coercive we have
  \begin{equation*}
    \begin{array}{rcl}
      \underline{c} \DGnorm{ u_i }{}^{2} \leq \tilde{a}_i(u_i,u_i) &=& 
      - \tilde{a}_{i\GAM}(u_i,\varphi)
      \\
      &=& \dotS{\mu u_i - \PDif{u_i}{\NOR_i} }{ \varphi }_{\GAM}
      \\
      &\leq& \mu \norm{ u_i }_{\GAM} \norm{ \varphi }_{\GAM} + 
      \norm{ \PDif{u_i}{\NOR_i} }_{\GAM} \norm{ \varphi }_{\GAM}
      \\
      &\leq& \mu \norm{ u_i }_{\GAM} \norm{ \varphi }_{\GAM} + 
      \frac{C}{\sqrt{h}} \norm{ \nabla u_i }_{\Ti{i}} \norm{ \varphi }_{\GAM}
      \\
      &\leq& \mu^{\half} \left( 1 + \frac{C}{\sqrt{\alpha} } \right) 
      \DGnorm{ u_i }{} \cdot \norm{ \varphi }_{\GAM},
    \end{array}
  \end{equation*}
  which completes the proof with $C_{\Hopt} := \left( 1 +
  \frac{C}{\sqrt{\alpha} } \right)^{2} \cdot \frac{1}{\underline{c}^2}
  > 1 $.  \quad
\end{proof}

We see that $C_\Hopt > 1$, which does not provide a sharp estimate for
the maximum eigenvalue of $\tB_i$. We now show how to obtain a sharp
estimate for the maximum eigenvalue of the $\tB_i$. Recall that the
global matrix $\tA$ is s.p.d.~and the positive definiteness is proved
by using for each subdomain $\half \tA_\GAM$ in
(\ref{eq:coercAtilde}). Therefore we consider the
s.p.d.~matrix
\begin{equation*}
  \hat{A} := \MATT{\tA_i}{\tA_{i\GAM}}{\tA_{\GAM i}}{\half \tA_\GAM}.
\end{equation*}
To show positive-definiteness, let $\wB:=(\uB_i, \phiB)^\top$ and
observe
\begin{equation} \label{eq:hatAspd}
  \wB^\top \hat{A} \wB = \tilde{a}_i(u_i,u_i) + 2
  \tilde{a}_{i\GAM}(u_i,\varphi) + \half
  \tilde{a}_\GAM(\varphi,\varphi) \geq c \Bnorm{(u_i,\varphi)}{,i}^2,
\end{equation}
for all $u_i \in V_{h,i}$ and $\varphi \in \Lambda_h$. Now let $u_i =
\Hopt_i(\varphi)$, then by a simple manipulation we have $ \phiB^\top
\big( \half \tA_\GAM - \tB_i \big) \phiB = \wB^\top \hat{A}
\wB$. Combining with (\ref{eq:hatAspd}) and recalling that $
\phiB^\top \tA_\GAM \phiB = 2 \mu \norm{\varphi}^2_\GAM $ we obtain
\begin{equation} \label{eq:Bsharp}
  \mu \norm{\varphi}_\GAM^2 - c \Bnorm{(\Hopt_i(\varphi),\varphi)}{,i}^2 \geq
  \phiB^\top \tB_i \phiB.
\end{equation}
This gives a sharp estimate for the maximum eigenvalue of $\tB_i$ if we
can bound the second term from below which is stated in the following
lemma.
\begin{lemma} \label{lemma:HDGlowerforB}
Let $\varphi \in \Lambda_h$ and $u_i \in V_{h,i}$ for $i=1,2$. Let
$H_i$ be the diameter of the subdomain. Then we have
\begin{equation*}
  \frac{c}{H_i} \norm{ \varphi }_\GAM^2 \leq 
  \Bnorm{(u_i,\varphi)}{,i}^{2}.
\end{equation*}
\end{lemma}
\begin{proof}
We first invoke triangle inequality and then Young's inequality
\begin{equation*}
   \norm{ \varphi }_\GAM^2 \leq 
   \norm{ u_i - \varphi }_\GAM^2 + \norm{ u_i }_\GAM^2 \leq 
  {H_i} h^{-1} \norm{ u_i - \varphi }_\GAM^2 +  \norm{ u_i }_\GAM^2, 
\end{equation*}
where the last inequality is due to the fact that $h \leq {H_i}$. Now
for the second term on the right-hand side we apply the trace
inequality from Lemma \ref{lemma:karakashian}, and subsequently the
Poincar\'e inequality from Lemma \ref{lemma:poincare} with $\nu =
\partial \OM_i \setminus \GAM $. We obtain
\begin{equation*}
\begin{array}{rcl}
  \norm{ \varphi }_\GAM^2 &\leq&
  {H_i} h^{-1} \norm{ u_i - \varphi }_\GAM^2 +
  c_1 H_i \big(
  \norm{ \nabla u_i }_{\OM_i}^2 + h^{-1} \norm{ \jump{u_i} }_{\EPS_i \setminus
    \GAM}^2
  \big)
  \\
  &\leq& c_2 H_i \Bnorm{(u_i,\varphi)}{,i}^2,
\end{array}
\end{equation*}
which completes the proof. \quad
\end{proof} 

We are now in the position to prove the estimate for the eigenvalues
of $\tB_i$.
\begin{lemma} \label{lemma:Beigen}
There exists $\alpha>0$, sufficiently large, such that 
\begin{equation*}
  c_B \mu \norm{\varphi}_\GAM^2 \leq \phiB^\top \tB_i \phiB \leq \Big( 1
  - C_B \frac{h}{H \alpha} \Big) \mu \norm{\varphi}_\GAM^2,
  \quad \forall \varphi \in \Lambda_h,
\end{equation*}
where $0<c_B<1$. Therefore $\tB_i$ is s.p.d.~Moreover $\tA_\GAM - 2
\tB_i$ is s.p.d.
\end{lemma}
\begin{proof}
To show the proof of the lower bound we use (\ref{eq:BtoA}),
coercivity of $\tilde{a}_i(\cdot,\cdot)$ and Lemma
\ref{lemma:trinvphi} to obtain
\begin{equation*}
  \phiB^\top \tB_i \phiB = \tilde{a}_i(u_i,u_i) \geq \underline{c}
  \DGnorm{u_i}{}^2 \geq \underline{c} \cdot c_\Hopt \cdot \mu \norm{\varphi}_\GAM^2.
\end{equation*}
This completes the lower bound by setting $c_B := \underline{c} \cdot
c_\Hopt < 1$. For the upper bound we use inequality (\ref{eq:Bsharp})
and Lemma \ref{lemma:HDGlowerforB} where we obtain
\begin{equation*}
  \phiB^\top \tB_i \phiB \leq \Big( 1 - \frac{c}{H} \frac{h}{\alpha}
  \Big) \mu \norm{\varphi}_\GAM^2.
\end{equation*}
Finally from inequality (\ref{eq:Bsharp}) we have that $\tA_\GAM - 2
\tB_i$ is s.p.d.
\end{proof}
\begin{remark}
This estimate shows that the condition number satisfies
\begin{equation*}
  \kappa(\tB_i) \leq c_B^{-1} \big( 1 - C_B \frac{h}{H \alpha } \big),
\end{equation*}
which implies that $\tB_i$ is {\normalfont scalable}. In other words if
we keep the ratio $h/H$ constant the condition number does not
change. Geometrically that is equivalent of scaling the subdomain and
the triangulation at the same rate which does not change the
{\normalfont entries} of the $\tB_i$ nor its {\normalfont
  size}. Therefore the condition number of $\tB_i$ is expected not to
change.
\end{remark}
\subsection{Eigenvalue estimate for $\tilde{S}_\GAM$}
Estimating eigenvalues of the Schur complement is similar to
estimating eigenvalues of $\tB_i$. To show the lower bound in estimate
(\ref{eq:Sest}), we need the following lemma.
\begin{lemma} \label{lemma:HDGlower}
Let $\varphi \in \Lambda_h$ and $u_i \in V_{h,i}$ for $i=1,2$. Let
$H_\OM$ be the diameter of the domain and $H$ be the maximum diameter of
the subdomains. Then we have
\begin{equation*}
  c \frac{H}{H_\OM^2} \norm{\varphi}_\GAM^2 \leq 
  \sum_{i=1}^2 \Bnorm{(u_i,\varphi)}{,i}^{2}.
\end{equation*}
\end{lemma}
\begin{proof}
First we invoke a triangle inequality 
\begin{equation*}
  H_i \norm{ \varphi }_\GAM^2 \leq 
  H_i \norm{ u_i - \varphi }_\GAM^2 + H_i \norm{ u_i }_\GAM^2 \leq 
  {H_i^2} h^{-1} \norm{ u_i - \varphi }_\GAM^2 + H_i \norm{ u_i }_\GAM^2, 
\end{equation*}
where the last inequality is due to the fact that $h \leq {H_i}$. Now for
the second term on the right-hand side, observe that using Lemma
\ref{lemma:karakashian} we have
\begin{equation*}
  c_i H_i \norm{ u_i }_{\GAM}^{2} \leq
  c_i H_i \norm{ u_i }_{\partial \OM_i}^{2} \leq 
  \norm{u_i}_{\OM_i}^{2} + H_i^{2} \big( \norm{ \nabla u_i }_{\OM_i}^{2} + 
  h^{-1} \norm{\jump{u_i}}_{\EPS_i \setminus \partial \OM_i}^{2} \big).
\end{equation*}
We sum over both subdomains and
invoke Lemma \ref{lemma:poincare} for the $\Lsp^2$-norm of $u$ over
$\OM$
\begin{equation*}
\begin{array}{rcl}
  c H \sum_{i=1}^2 \norm{u_i}_{\GAM}^2
  &\leq& \norm{ u }_\OM^2 + H^2 \sum_{i=1}^2 
  \big(
  \norm{ \nabla u_i }_{\OM_i}^{2} + h^{-1} \norm{\jump{u_i}}_{\EPS_i
    \setminus \partial \OM_i}^{2}
  \big)
  \\
  &\leq& C H_\OM^2 \big( 
  \norm{ \nabla u }_{\OM}^2 
  + h^{-1} \norm{ \jump{ u } }_{\EPS \setminus \partial \OM}^2
  + h^{-1} \norm{ { u } }_{ \partial \OM}^2 \big)
  \\
  && \quad
  +
  H^2 \sum_{i=1}^2 \big(
  \norm{ \nabla u_i }_{\OM_i}^{2} + h^{-1} \norm{\jump{u_i}}_{\EPS_i
    \setminus \partial \OM_i}^{2}
  \big).
\end{array}
\end{equation*}
Noting that $H \leq H_\OM$ and by definition of
$\Bnorm{(u_i,\varphi)}{,i}$ we obtain
\begin{equation*}
\begin{array}{rcl}
  c H \sum_{i=1}^2 \norm{u_i}_{\GAM}^2
  &\leq& H_\OM^{2}
  \sum_{i=1}^2 \big( 
  \norm{ \nabla u_i }_{\OM_i}^2
  + h^{-1} \norm{\jump{u_i}}_{\EPS_i
    \setminus \partial \OM_i}^{2}
  + h^{-1} \norm{{u_i}}_{ \partial \OM
    \cap \partial \OM_i}^{2}
  \big)
  \\
  && + H_\OM^{2} h^{-1} \norm{\jump{u}}_{\GAM}^{2}
  \\
  &\leq& H_\OM^{2}
  \sum_{i=1}^2 \big( 
  \norm{ \nabla u_i }_{\OM_i}^2
  + h^{-1} \norm{\jump{u_i}}_{\EPS_i
    \setminus \partial \OM_i}^{2}
  + h^{-1} \norm{{u_i}}_{ \partial \OM
    \cap \partial \OM_i}^{2}
  \big)
  \\
  && + H_\OM^{2} h^{-1} \big( 
  \norm{u_1 - \varphi }_{\GAM}^{2} 
  + \norm{u_2 - \varphi }_{\GAM}^{2}
  \big)
  \\
  &\leq& H_\OM^2 \sum_{i=1}^2 \Bnorm{(u_i,\varphi)}{,i}^2.
\end{array}
\end{equation*}
Substituting back into the first inequality completes the proof. \quad
\end{proof}
\begin{lemma} \label{lemma:shurpos}
There exists $\alpha>0$, sufficiently large, such that
  \begin{equation*}
    c \frac{H}{H_\OM^2} \norm{\varphi}_{\GAM}^{2} \leq \phiB^\top \tilde{S}_{\GAM} \phiB 
    \leq \frac{2 \alpha}{h} \norm{\varphi}_{\GAM}^2,
  \end{equation*}
Therefore $\tilde{S}_\GAM$ is s.p.d. Moreover $\tA_\GAM - \tB_i$ is s.p.d.
\end{lemma}
\begin{proof}
The symmetry is easy to check since $\tA_\GAM$ and $\tB_1$, $\tB_2$ are
symmetric.  For the upper bound in the estimate we recall that
$\tB_1$, $\tB_2$ are positive definite and hence
\begin{equation*}
\phiB^\top \tilde{S}_{\GAM} \phiB = \phiB^\top (\tA_\GAM - \sum_{i=1}^2 \tB_i )
\phiB \leq \phiB^\top \tA_{\GAM} \phiB = 2\mu \norm{\varphi}_\GAM^{2}.
\end{equation*}
Now let $u_i := \Hopt_i(\varphi)$ and $\vB := ( \uB_1, \uB_2, \phiB
)^\top$. A straightforward calculation shows that $ \phiB^\top
\tilde{S}_\GAM \phiB = \vB^\top \tA \vB$. Then the coercivity of the
bilinear form $\tilde{a}(\cdot,\cdot)$ and an application of Lemma
\ref{lemma:HDGlower} yields
\begin{equation*}
  \phiB^\top \tilde{S}_\GAM \phiB = \vB^\top \tA \vB 
  \equiv \tilde{a}( (u,\varphi), (u,\varphi) ) 
  \geq c \sum_{i=1}^2 \Bnorm{(u_i,\varphi)}{,i}^{2}
  \geq c \frac{H}{H_\OM^2} \norm{ \varphi }_\GAM^2.
\end{equation*}
For the final statement, observe that for all $\varphi \not = 0$ we have
\begin{equation*}
  \phiB^\top ( \tA_\GAM - \tB_i ) \phiB > \phiB^\top (\tA_\GAM - \sum_{j=1}^2 \tB_j) \phiB
  > 0,
\end{equation*}
since the $\{ \tB_i \}$ are positive definite. 
This completes the proof. \quad
\end{proof}
\begin{remark}
  Note that Lemma \ref{lemma:shurpos} provides an upper bound for the
  condition number,
  $
  \kappa(\tilde{S}_\GAM) \leq O(\frac{\alpha}{h})
  $.
  A similar result also holds for classical FEM,
  see \cite{brenner} and \cite[Lemma 4.11]{widlund}.
\end{remark}

\section{Schwarz methods and the Schur complement} \label{sec:schwarz}

In order to solve the Schur complement system we can devise a Schwarz
method to obtain $\lambda_h$. We will prove that a natural Schwarz
method for the Schur complement is equivalent to the block Jacobi
iteration in (\ref{eq:blockJacobi}), but it suffers from slow
convergence. Later we show how to obtain an optimized Schwarz method
for the Schur complement which converges much faster to the same fixed
point.

Let us relax the constraint that $\lambda_h$ is single-valued. Let
$\lambda_{h,1}, \lambda_{h,2} \in \Lambda_h $.  Assume $\lambda_{h,2}$
is known; that is we know $u_2 \in V_{h,2}$.  Then we can split the
Schur complement system (\ref{eq:shur}) and obtain an approximation
for $\lambda_{h,1}$ and consequently $u_{1}\in V_{h,1}$ from
\begin{equation*}
  (\tA_\GAM - \tB_1 ) \lB_1 = \tB_2 \lB_2 + \BO{g}_{\GAM}.
\end{equation*}
As a consequence of Lemma \ref{lemma:shurpos}, $(\tA_\GAM - \tB_1 )$
is invertible and we can obtain $\lambda_{h,1}$. This suggests an
iterative method to obtain $\lambda_h$. We will see that this produces
identical iterates as the block Jacobi method.

\begin{algorithm} \label{algo:addschwarz}
Let $\lambda_{h,1}^{(0)}, \lambda_{h,2}^{(0)} \in \Lambda_h$ be two
random initial guesses.  Then for $n=1,2,\hdots$ find
$\big\{ \lambda_{h,i}^{(n)} \big\}$ such that
\begin{equation} \label{eq:shurITE}
  \begin{array}{rcl}
    (\tA_\GAM^{} - \tB_1^{} ) \lB_1^{(n)} = \tB_2^{} \lB_2^{(n-1)} + \BO{g}_{\GAM}^{},
    \\
    (\tA_\GAM^{} - \tB_2^{} ) \lB_2^{(n)} = \tB_1^{} \lB_1^{(n-1)} + \BO{g}_{\GAM}^{}.	
  \end{array}
\end{equation}
\end{algorithm}
%
%
At convergence, we have $\tA_{\GAM}(\lB_1 - \lB_2)=0$ which implies
$ \lB_1 = \lB_2 = \tilde{S}_{\GAM}^{-1} \BO{g}_{\GAM}$.

The following result shows that the above method generates the same
iterates as the block Jacobi iteration (\ref{eq:blockJacobi}). By
linearity it suffices to consider the error equation, $f=0$, which
implies $\BO{g}_\GAM = 0$.
\begin{proposition}
Let $\lambda_{h,1}^{(0)}, \lambda_{h,2}^{(0)}$ be two random initial
guesses of Algorithm \ref{algo:addschwarz} and without loss of
generality suppose $f=0$. Set the initial guess of the block Jacobi
iteration (\ref{eq:blockJacobi}) to be $ u_i^{(0)} =
\Hopt_i^{}(\lambda_{h,i}^{(0)}) $. Then $ u_i^{(n)} =
\Hopt_i^{}(\lambda_{h,i}^{(n)}) $ for all $n>0$, i.e.~both methods
produce the same iterates.
\end{proposition}
\begin{proof}
  See \cite{hajian2014}.
\end{proof}
%

\subsection{Analysis of classical Schwarz for the Schur complement}

By linearity we consider the error equations and we denote by $
\BO{e}_{i}^{(n)} := \lB_{i}^{(n)} - \lB$.  The iterations in
\eqref{eq:shurITE} can be rewritten in a more suitable form for
analysis. Since $\tA_{\GAM}$ is s.p.d.~(it is just a scaled mass
matrix), the square-root $\tA_{\GAM}^{1/2}$ exists and is also
s.p.d. Therefore, for $i,j \in \{1,2\}$ and $i\not = j$ we can write
equivalently
\begin{equation*}
  \begin{array}{lrcl}
    &
    (\tA_\GAM^{} - \tB_i^{} ) \BO{e}_i^{(n)} &=& \tB_j^{} \BO{e}_j^{(n-1)}
    \\
    \Leftrightarrow &
    \tA_{\GAM}^{1/2} ( I - \tA_{\GAM}^{-1/2} \tB_i^{} \tA_{\GAM}^{-1/2} )
    \tA_{\GAM}^{1/2} \BO{e}_i^{(n)} 
    &=& \tB_j^{} \BO{e}_j^{(n-1)}
    \\
    \Leftrightarrow &
    ( I - \tA_{\GAM}^{-1/2} \tB_i^{} \tA_{\GAM}^{-1/2} ) \tilde{\BO{e}}_i^{(n)} &=& 
    (\tA_{\GAM}^{-1/2} \tB_j^{} \tA_{\GAM}^{-1/2}) \tilde{\BO{e}}_j^{(n-1)},
  \end{array}
\end{equation*}
where $ \tilde{\BO{e}}_{i} = \tA_{\GAM}^{1/2} \BO{e}_i^{} $. We define
\begin{equation}
  C_i := \tA_{\GAM}^{-1/2} \tB_i^{} \tA_{\GAM}^{-1/2},
\end{equation}
which is invertible and symmetric. Since $\tA_\GAM - \tB_i$ is invertible and
$\tA_\GAM^{1/2}$ exists we can conclude that $I-C_i$ is also invertible by
definition. Therefore we have
\begin{equation*}
    ( I - C_i^{} ) \tilde{\BO{e}}_{i}^{(n)} = C_j^{} \tilde{\BO{e}}_{j}^{(n-1)}
    = C_j^{} ( I - C_j^{} )^{-1} C_i^{} \tilde{\BO{e}}_{i}^{(n-2)},
\end{equation*}
or
\begin{equation*}
    \phiB_{i}^{(n)} = 
    C_j^{} ( I - C_j^{} )^{-1} \cdot C_i^{} ( I - C_i^{} )^{-1} \phiB_{i}^{(n-2)},
\end{equation*}
where $\phiB_i = ( I - C_i ) \tilde{\BO{e}}_{i}$. Finally the
iterations can be rewritten as
\begin{equation} \label{eq:shurITEan}
  \phiB_{i}^{(n)} = 
  ( C_j^{-1} - I )^{-1} \cdot ( C_i^{-1} - I )^{-1} \phiB_{i}^{(n-2)}.
\end{equation}

We show how the contraction factor of the iteration in
\eqref{eq:shurITEan} is related to the eigenvalues of $\{C_i\}$. Let
$\norm{\cdot}_{2}$ be the usual 2-norm in $\RE^{n}$, and denote by $D_i := (
C_i^{-1} - I )^{-1}$. Then we can estimate
\begin{equation*}
  \norm{\phiB_{i}^{(n)}}_2 \leq 
  \norm{D_j D_i}_2 \, \norm{\phiB_{i}^{(n-2)} }_2 \leq
  \norm{D_j}_{2} \, \norm{D_i}_2 \, \norm{\phiB_{i}^{(n-2)} }_2 =
  \rho( D_j ) \, \rho( D_i ) \, \norm{\phiB_{i}^{(n-2)} }_2,
\end{equation*}
since $\{ D_i \}$ are symmetric.  In other words we have used a
different norm for the error: with $E_i := (I - C_i^{} )
\tA_{\GAM}^{1/2}$, which is invertible, we have
\begin{equation*}
  \norm{ \BO{e}_i }_{E_i^\top E_i^{}} = \norm{ E_i \BO{e}_i }_{2}
  = \norm{ \phiB_i }_{2}.
\end{equation*}
Let $\sigma(M)$ denote an eigenvalue of a given matrix $M$. Then we have
\begin{equation*}
  \rho(D_i) := \max_{\sigma(D_i)} | \sigma(D_i) | = 
  \max_{\sigma(C_i)} \left| \frac{\sigma(C_i)}{1- \sigma(C_i)} \right|.
\end{equation*}
Hence a sufficient condition for convergence is that $ \sigma(C_i) \in
(-\infty, 1/2) $. On the other hand by definition of $C_i$ we know
that $\sigma(C_i)$ are the eigenvalues of the generalized eigenvalue
problem $\tB_{i} \phiB = \sigma \, \tA_\GAM \phiB $. Since both
$\tA_\GAM$ and $\tB_i$ are s.p.d.,~$\sigma(C_i)$ is
positive.
Therefore a sufficient condition for convergence is to show that $
\sigma(C_i) \in (0, 1/2) $.

Recall that since $C_i$ is symmetric we have
\begin{equation}
  \label{eq:CestL}
  \sigma_{\text{min}}(C_i) = 
  \inf_{\phiB \not = 0} \frac{\phiB^\top \tB_i \phiB}{\phiB^\top \tA_{\GAM} \phiB}
  = 
  \inf_{\phiB \not = 0} \frac{\phiB^\top \tB_i \phiB}{2\mu
    \norm{\varphi}_{\GAM}^{2}}
  \geq \frac{c_B}{2},
\end{equation}
where we have used the lower bound estimate of Lemma
\ref{lemma:Beigen}. Here $0 < c_B < 1$.
The upper bound for $\sigma_{\max}(C_i)$ can also be obtained using
Lemma \ref{lemma:Beigen}. Hence
\begin{equation}
  \label{eq:CestU}
  \sigma_{\max}(C_i) =
  \sup_{\phiB \not = 0} \frac{\phiB^\top \tB_i \phiB}{2\mu \norm{\varphi}_{\GAM}^{2}}
  \leq \frac{1}{2} \Big(1 - C \frac{h}{\alpha H} \Big),
\end{equation}
which is strictly less than $\frac{1}{2}$. Consequently for the
eigenvalues of $D_i$, we obtain the estimate
\begin{equation*}
  0 < \frac{c_B}{2 - c_B} \leq \sigma(D_i) \leq 
  1 - C \frac{h}{\alpha H} < 1.
\end{equation*}
We summarize the convergence result in the following theorem.
\begin{theorem}
There exists an $\alpha>0$ independent of $H$ and $h$ such that
Algorithm \ref{algo:addschwarz} converges and the contraction factor
is bounded by
\begin{equation}
  \rho \leq 1 - O(h).
\end{equation}
\end{theorem}

\subsection{Analysis of an optimized Schwarz method for the Schur complement}

As it has been shown in \cite{hajian2013block}, the IPH discretization
is imposing Robin transmission conditions between subdomains, and the
Robin parameter is precisely the penalty parameter $\mu$ of the DG
method. For approximation purposes and ensuring coercivity, $\mu$ is
set to be ${\alpha}{h^{-1}}$ for some $\alpha>0$ large and independent
of $h$.

In the Schwarz theory with Robin transmission conditions this choice
of $\mu$ corresponds to damping high frequencies of the DtN
operator. In other words, the low frequencies are responsible for the
slow convergence of the algorithm that we have analyzed in the
previous subsection; as we have shown the contraction factor is
$\rho = 1 - O(h)$. Optimized Schwarz theory suggests to choose
the Robin parameter $O(h^{-1/2})$, see \cite{ganderos}, while this is
not possible for an IPH discretization since we lose coercivity and
optimal approximation properties.

The remedy comes from an idea first introduced in
\cite{discacciati2004operator} and later independently in
\cite{dolean} for Maxwell's equations.  The idea is to perturb the
transmission conditions such that while iterating we produce a
different sequence but obtaining the same fixed-point as the original
Schwarz algorithm.

Let us introduce two new unknowns, one for each subdomain, along the
interface called $\{ r_{12}, r_{21} \}$ such that $r_{ij} \in
\Lambda_h$. Recall that by Proposition \ref{prop:Bopt} an application
of $\tB_i \lB_i$ is equivalent to $\mu u_i - \PDif{u_i}{\NOR_i}$ on
the interface where $u_i := \Hopt_i(\lambda_{h,i})$.  Now let $r_{ij}
= (\mu u_j - \PDif{u_j}{\NOR_j} ) |_{\GAM}$. Let us denote by $M_\GAM$
the mass matrix along the interface and $\BO{r}_{ij}$ the
corresponding DOFs of $r_{ij}$. Then we observe that
\begin{equation*}
\phiB^\top M_\GAM \BO{r}_{ij} = \dotS{r_{ij} }{\varphi}_{\GAM}
= \dotS{ \mu u_j - \PDif{u_j}{\NOR_j} }{\varphi}_{\GAM} 
= \phiB^{\top} \tB_{j} \lB_{j}, 
\quad \forall \varphi \in \Lambda_h.
\end{equation*}
Therefore we conclude that
\begin{equation*}
  M_\GAM \BO{r}_{ij} = \tB_j \lB_j,
\end{equation*}
and the Schwarz iteration \eqref{eq:shurITE} can be rewritten as
\begin{equation*}
  \begin{array}{rcl}
    (\tA_\GAM - \tB_i ) \lB_i^{(n)} &=& M_\GAM \BO{r}_{ij}^{(n)} + \BO{g}_{\GAM},
    \\
    M_\GAM \BO{r}_{ij}^{(n)} &=& \tB_j \lB_j^{(n-1)}.
  \end{array}
\end{equation*}
We modify the second equation as suggested in \cite{dolean} and
\cite{hajian2014} to the form
\begin{equation*}
  M_\GAM^{} \BO{r}_{ij}^{(n)} - \hat{p}\, \tB_i^{} \lB_i^{(n)}  = 
  \tB_j^{} \lB_j^{(n-1)} - \hat{p}\, M_\GAM^{} \BO{r}_{ji}^{(n-1)},
\end{equation*}
for $i,j \in \{1,2\}$ and $i\not = j$. Here $0 \leq \hat{p} < 1$ is a
parameter which we use for optimization.  At convergence one recovers
the original equations and therefore the fixed point of the iteration
is the same as for the original method.

\begin{remark}
\label{remark:phat}
The above modification is shown in \cite{hajian2014} to be equivalent
(at the continuous level) to imposing
\begin{equation}
  \Big( \frac{1 - \hat{p}}{1 + \hat{p}} \, \mu + \PDif{}{\NOR_i} \Big) u_i^{(n)} =
  \Big( \frac{1 - \hat{p}}{1 + \hat{p}} \, \mu + \PDif{}{\NOR_i} \Big) u_j^{(n-1)}
\end{equation}
for $i,j \in \{1,2\}$ and $i\not=j$. Note that if $\hat{p} =
\frac{1-\sqrt{h}}{1+\sqrt{h}}$ then $\frac{1 - \hat{p}}{1 + \hat{p}}
\mu \propto \frac{1}{\sqrt{h}}$ which is the right choice of parameter
according to optimized Schwarz theory. We will see that this is
exactly the right choice for $\hat{p}$ at the discrete level.
\end{remark}

The analysis of this algorithm is possible using the
framework established for the original method. We can eliminate the
$\{ r_{ij} \}$ as follows:
\begin{equation*}
  \begin{array}{rcll}
    (\tA_\GAM^{} - \tB_i^{} ) \lB_i^{(n)} &=& 
    \hat{p} \tB_i^{} \lB_i^{(n)} + \tB_j^{} \lB_j^{(n-1)} - 
    \hat{p} M_\GAM^{} \BO{r}_{ji}^{(n-1)}
    & + \BO{g}_{\GAM}^{} 
    \\
    &=&
    \hat{p} \tB_i^{} \lB_i^{(n)} + \tB_j^{} \lB_j^{(n-1)} 
    - \hat{p} (\tA_\GAM^{} - \tB_j^{} ) \lB_j^{(n-1)}
    & + (1+\hat{p})\BO{g}_{\GAM}^{} ,
  \end{array}	
\end{equation*}
which simplifies to
\begin{equation*}
  (\tA_\GAM^{} - (1+\hat{p})\tB_i^{} ) \lB_i^{(n)} = 
  - ( \hat{p} \tA_\GAM^{} - (1+\hat{p})\tB_j^{} ) \lB_j^{(n-1)}
  + (1+\hat{p}) \BO{g}_{\GAM}^{}.
\end{equation*}
\begin{algorithm} \label{algo:OS}
Let $\lambda_{h,1}^{(0)}, \lambda_{h,2}^{(0)} \in \Lambda_h$ be two
random initial guesses.  Then for $n=1,2,\hdots$ find
$\big\{ \lambda_{h,i}^{(n)} \big\}$ such that
\begin{equation} \label{eq:OSITE}
  \begin{array}{rcl}
    (\tA_\GAM - (1+\hat{p})\tB_1 ) \lB_1^{(n)} &=& 
    - ( \hat{p} \tA_\GAM - (1+\hat{p})\tB_2 ) \lB_2^{(n-1)}
    + (1+\hat{p}) \BO{g}_{\GAM} ,
    \\
    (\tA_\GAM - (1+\hat{p})\tB_2 ) \lB_2^{(n)} &=& 
    - ( \hat{p} \tA_\GAM - (1+\hat{p})\tB_1 ) \lB_1^{(n-1)}
    + (1+\hat{p}) \BO{g}_{\GAM} .
  \end{array}
\end{equation}
\end{algorithm}

Since $\hat{p}<1$, we can use Lemma \ref{lemma:Beigen} and
conclude that the left hand side is positive definite and therefore
invertible. At convergence we have $(1-\hat{p})\tA_\GAM ( \lB_1 -
\lB_2 ) = 0 $ which implies $\lB_1=\lB_2 = \tilde{S}_\GAM^{-1}
\BO{g}_\GAM$ if $\hat{p} \not = 1$.

Comparing to the original Schwarz method, Algorithm
\ref{algo:addschwarz}, we weakened the positive-definiteness of the
left-hand side.  This plays a key role in faster convergence.  The
optimized algorithm can be viewed as a different splitting of the
Schur complement. More precisely we multiplied it by $(1+\hat{p})$ and
this time a fraction of $\tA_\GAM$, namely $\hat{p} \tA_\GAM$, has
been put to the right-hand side.

We consider the error equation and we can proceed as before to obtain
an iteration for $\BO{e}_i$ only,
\begin{equation*}
  \begin{array}{ll}
    & (\tA_\GAM - (1+\hat{p})\tB_i ) \BO{e}_i^{(n)}  =
    \\
    & \qquad  
    ( \hat{p} \tA_\GAM - (1+\hat{p})\tB_j ) \cdot 
    (\tA_\GAM - (1+\hat{p})\tB_j )^{-1} \cdot 
    ( \hat{p} \tA_\GAM - (1+\hat{p})\tB_i ) 
    \BO{e}_i^{(n-2)}.
  \end{array}
\end{equation*}
With $\phiB_i = ( I - (1+\hat{p})C_i ) \tA_\GAM^{1/2} \BO{e}_i$, we have
\begin{equation*}
  \begin{array}{ll}
    &\phiB_i^{(n)} = ( \hat{p} I - (1+\hat{p}) C_j ) \cdot
    ( I - (1+\hat{p})C_j )^{-1}
    \\
    & 
    \qquad \qquad
    \cdot
    ( \hat{p} I - (1+\hat{p}) C_i ) \cdot
    ( I - (1+\hat{p})C_i )^{-1} \phiB_{i}^{(n-2)}.
  \end{array}
\end{equation*}
Denoting by $\hat{D}_i := ( \hat{p} I - (1+\hat{p}) C_i ) \cdot ( I -
(1+\hat{p})C_i )^{-1} $ and simplifying, we get
\begin{equation}
  \hat{D}_i = I - (1-\hat{p}) \big( I - (1+\hat{p})C_i \big)^{-1},
\end{equation}
which shows that $\hat{D}_i$ is symmetric. Therefore we have
\begin{equation*}
  \norm{\phiB_{i}^{(n)}}_2 \leq 
  \rho( \hat{D}_j ) \, \rho( \hat{D}_i ) \, \norm{\phiB_{i}^{(n-2)} }_2.
\end{equation*}
The estimate for the eigenvalues of $\hat{D}_i$ can be obtained as
before. More precisely we have
\begin{equation*}
  \sigma( \hat{D}_i ) = 1 - \frac{1-\hat{p}}{1 - (1+\hat{p}) \,
    \sigma(C_i)}.
\end{equation*}
Recall that $\sigma(C_i) \in \left[\frac{1}{2} - c, \frac{1}{2} - C
  \frac{h}{\alpha H} \right]$ for $0< c < \frac{1}{2}$ and $C>0$ and
independent of $h$, $H$.  We can use $\hat{p}$ to optimize $\rho(
\hat{D}_i )$. Following Remark \ref{remark:phat}, let us make the ansatz
\begin{equation*}
  \hat{p} = \frac{1 - (\frac{h}{\alpha})^{\gamma}}{1 + (\frac{h}{\alpha})^{\gamma}} < 1,
  \quad \gamma \in \RE^+.
\end{equation*}
This implies that
\begin{equation}
  1 - \frac{1}{\half + { \frac{C}{H} } (\frac{h}{\alpha})^{1-\gamma} }  
  \leq \sigma(\hat{D}_i) \leq
  1 - \frac{1}{\half + c (\frac{h}{\alpha})^{-\gamma} },
\end{equation}
Best performance is achieved, if $\gamma := \half$
which as $h \rightarrow 0$ leads to
\begin{equation}
  -1 + c_1 \sqrt{ \frac{h}{\alpha} } \leq \sigma(\hat{D}_i) \leq
  1 - c_2 \sqrt{ \frac{ h}{\alpha} }.
\end{equation}
Note that the iteration matrix, $\hat{D}_i$, is not positive definite
anymore but it has a converging spectrum and the contraction factor is
much better than the one in Algorithm \ref{algo:addschwarz}.  We
summarize our results in 
\begin{theorem}
There exists an $\alpha>0$ independent of $H$ and $h$ such that
Algorithm \ref{algo:OS} converges and the contraction factor is
bounded by
\begin{equation}
  \rho \leq 1 - O(\sqrt{h}).
\end{equation}
\end{theorem}

\section{A multi subdomain algorithm}\label{sec:multi}

We have introduced and analyzed a two subdomain optimized Schwarz
method (OSM) so far. In this section we introduce a multi subdomain
algorithm for the IPH discretization. This algorithm is a natural
generalization of the two subdomain method. Often special care has to
be taken in OSMs for classical FEM discretizations at cross-points,
that is a node which is shared by more than two subdomains, see
\cite{gander2012best,gander2013applicability,gander2014CPDD}. This is
not the case when we work with a DG discretization because subdomains
communicate with each other only if they have an intersection of
non-zero measure. Therefore the problem with cross-points does not
arise, since a cross-point is of measure-zero, as at the continuous
level.

Let us start defining the multi-subdomain geometry. We first partition
the mono-domain $\Omega$ into $N_s$ subdomains such that the
interface, $\GAM$ between them is a subset of internal edges,
$\EPS^0$. More precisely, we denote the subdomains by $\{ \OM_i^{}
\}_{i=1}^{N_s}$ and the interface between two subdomain by
\begin{equation*}
  \GAM_{ij}^{} := \partial \OM_i \cap \partial \OM_j,
  \quad (i \not= j),
\end{equation*}
and the global interface by 
\begin{equation*}
  \GAM := \bigcup_{i\not=j} \GAM_{ij} \subset \EPS^0.
\end{equation*}

Now the hybridizable formulation of IPH can be written as: find
$(u_h,\lambda_h) \in V_h \times \Lambda_h$ such that 
\begin{equation}
  \tilde{a}( (u_h,\lambda_h),(v,\varphi) ) = \dotV{f}{v}_{\Th}, \quad
  \forall (v, \varphi) \in V_h \times \Lambda_h,
\end{equation}
where the bilinear form is defined as
\begin{equation}
  \tilde{a}( (u,\lambda), (v,\varphi) ) := \tilde{a}_\GAM(\lambda,
  \varphi) + \sum_{i=1}^{N_s} \big( \tilde{a}_i(u_i,v_i) +
  \tilde{a}_{i\GAM}(u_i,\varphi) + \tilde{a}_{i\GAM}(v_i,\lambda) \big).
\end{equation}
The only modified bilinear form is $\tilde{a}_{i\GAM}(\cdot,\cdot)$
  since it acts now on $\partial \OM_i$, that is
\begin{equation}
  \tilde{a}_{i\GAM}(u_i,\varphi) := 
    \dotS{ \PDif{u_i}{\NOR_i} - \mu u_i }{\varphi}_{\partial \OM_i}.
\end{equation}
\begin{figure} \label{fig:multiabs}
  \centering
  \epsfig{file=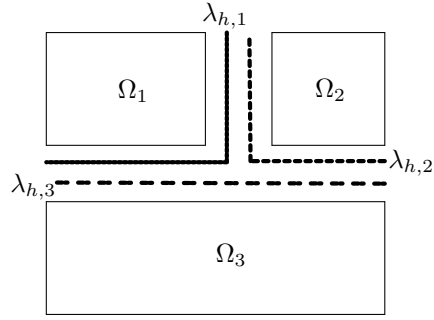, scale=1.0}
  \caption{A multi-subdomain configuration with an interface variable,
    $\{ \lambda_{h,i} \}$, assigned to each subdomain, $\OM_i$.
    }
\end{figure}

Let us focus on two subdomains which share an interface, $\GAM_{ij}$.
We observe that there are two sub-problems which are communicating
through $\lambda_h$ on $\GAM_{ij}$. That is
\begin{equation*}
  \begin{array}{rcl}
    \tilde{a}_i(u_i,v_i) + \tilde{a}_{i\GAM}(v_i,\lambda_h) &=& 
    \dotV{f}{v_i}_{\Ti{i}}, \quad \forall v_i \in V_{h,i},
    \\
    \tilde{a}_j(u_j,v_j) + \tilde{a}_{j\GAM}(v_j,\lambda_h) &=& 
    \dotV{f}{v_j}_{\Ti{j}}, \quad \forall v_j \in V_{h,j},
  \end{array}
\end{equation*}
and the continuity is imposed using
\begin{equation} \label{eq:contII}
  \lambda_h = \frac{1}{2\mu} \Big(\mu u_i - \PDif{u_i}{\NOR_i} \Big) + 
  \frac{1}{2\mu} \Big(\mu u_j - \PDif{u_j}{\NOR_j} \Big),
  \quad \text{on} \, \GAM_{ij}.
\end{equation}
Now we relax the constraint that $\lambda_h$ is single-valued on
$\GAM$ and allocate $\lambda_{h,i}$ to each subdomain $\OM_i$. Each
$\lambda_{h,i}$ is defined on $\partial \OM_i \setminus \partial \OM$;
for an example see Figure \ref{fig:multiabs}. We have therefore twice DOFs along
$\GAM_{ij}$. Therefore we should split the continuity equation
(\ref{eq:contII}) to provide two conditions; one for each
$\lambda_{h,i}$. We use the same idea as in Algorithm \ref{algo:OS}
and relax the continuity equation in the same fashion:
\begin{equation*}
    \frac{1}{1 + \hat{p}}  \lambda_{h,i} + 
    \frac{\hat{p}}{1 + \hat{p}}  \lambda_{h,j} = 
  \frac{1}{2\mu} \Big(\mu u_i - \PDif{u_i}{\NOR_i} \Big) + 
  \frac{1}{2\mu} \Big(\mu u_j - \PDif{u_j}{\NOR_j} \Big),
  \quad (i\not=j).
\end{equation*}
Here $\hat{p}$ is a parameter which is used for optimization
purposes. This suggests the following iterative method to find the
pairs $\big\{(u_i,\lambda_{h,i})\big\}_{i=1}^{N_s}$ in parallel:
\begin{algorithm} \label{algo:multi}
Let $\big\{(u_i^{(0)},\lambda_{h,i}^{(0)})\big\}_{i=1}^{N_s}$ be a set
of initial guesses for all subdomains. Then for $n=1,2,\hdots$ find
$\big\{(u_i^{(n)},\lambda_{h,i}^{(n)})\big\}_{i=1}^{N_s}$ such that
\begin{equation} \label{eq:multi-loc}
  \begin{array}{rcl}
    \tilde{a}_i(u_i^{(n)},v_i) +
    \tilde{a}_{i\GAM}(v_i,\lambda_{h,i}^{(n)} ) &=& 
    \dotV{f}{v_i}_{\Ti{i}}, \quad \forall v_i \in V_{h,i},
  \end{array}
\end{equation}
and the continuity condition on $\GAM_{ij}$ reads
\begin{equation} \label{eq:multi-cont}
  \frac{1}{1 + \hat{p}} \lambda_{h,i}^{(n)} -
  \frac{1}{2\mu} \Big(\mu u_i - \PDif{u_i}{\NOR_i} \Big)^{(n)}
  =
  - \frac{\hat{p}}{1 + \hat{p}} \lambda_{h,j}^{(n-1)} 
  + \frac{1}{2\mu} \Big(\mu u_j - \PDif{u_j}{\NOR_j} \Big)^{(n-1)}.
\end{equation}
\end{algorithm}

At convergence we obtain $(1- \hat{p})(\lambda_{h,i} - \lambda_{h,j})
= 0$. Therefore if $\hat{p} \not = 1$, we recover that $\lambda_h$ is
single valued.
%
%
%
\begin{remark}
  We can make an ansatz for the optimal choice of $\hat{p}$ similar
  to the two-subdomain case. The transmission condition
  (\ref{eq:contII}) can be viewed as a Robin transmission condition at
  the continuous level. The Robin parameter is $\mu^\star :=
  \frac{1-\hat{p}}{1+\hat{p}} \mu$. In order to converge fast we
  should set $\mu^\star = O(h^{-1/2})$. This corresponds to the choice
  $\hat{p} := \frac{1-\sqrt{h}}{1+\sqrt{h}} < 1$.
\end{remark}

\subsection{OSM as a preconditioner}

We show now how one can use OSM as a
preconditioner for a Krylov subspace method. We start by writing
Algorithm \ref{algo:multi} at the algebraic level. We first partition
the DOFs associated with $u_h \in V_h$ into
\begin{equation*}
  \uB := ( \uB_1, \uB_2, \hdots, \uB_{N_s} )^\top.
\end{equation*}
Then we form DOFs associated to the interface unknowns $\{
\lambda_{h,i} \}_{i=1}^{N_s}$ by
\begin{equation*}
  \frL  := ( \lB_{1}, \lB_{2}, \hdots, \lB_{N_s} )^\top,
\end{equation*}
and define the augmented DOFs by $\wB := ( \uB, \frL )^\top $.

Algorithm \ref{algo:multi} can be written at the algebraic level
as
\begin{equation}
  \begin{array}{rcl}
    \underbrace{
    \MATT{K_{u u}}{K_{u \ell}}{K_{\ell u}}{K_{\ell \ell}}
    }_{K}
    \wB^{(n)}
    = 
    \underbrace{
    \MATT{0}{0}{L_{\ell u}}{L_{\ell \ell}}
    }_{L}
    \wB^{(n-1)}
    + 
    \underbrace{\Arr{\fB}{0}}_{\BO{g}}.
  \end{array}
  \label{eq:multi-alg}
\end{equation}
Note that the left-hand side matrix $K$ consists of block matrices
communicating only with each pair $(u_{h,i},\lambda_{h,i})$. Therefore
we can ``invert'' subdomain blocks independently and in parallel. This gives
a parallel preconditioner for a Krylov subspace method applied to the
system $ (K-L) \wB = \BO{g} $. 

Since the stationary iterates (\ref{eq:multi-alg}) converge with the
contraction factor $\rho \leq 1 - O(\sqrt{h})$, we expect that a
preconditioned Krylov subspace method achieves another square-root in
the contraction factor, that is $\rho \leq 1 - O(h^{1/4})$. This is
observed in the numerical experiments. Therefore this is a more
attractive method compared to the CG method with an additive Schwarz
preconditioner which has the contraction factor $\rho \leq 1 -
O(\sqrt{h})$.

\section{Numerical experiments} \label{sec:num}

We perform numerical experiments on the model problem
\begin{equation}
  \begin{array}{rcll}
    (\eta-\Delta) u &=& f,\quad & \textrm{in $\OM$},\\
    u &=& 0, & \textrm{on $\partial \OM$},
  \end{array}
\end{equation}
where $\eta=1$ and $\OM$ is either a unit square, i.e.~$(0,1)^2$, or
an L-shaped (non-convex) domain. The interface is such that it does not
cut through any element, therefore $\GAM \subset \EPS$. We use $\PO^1$
elements and $\alpha = c {(k+1)(k+2)}$ where $c>0$ is a constant
independent of $h$ and $k=1$ (polynomial degree). The algorithms are
implemented using a FORTRAN90 library for DG methods called {\tt
  GDG90}. The codes are accessible at
\begin{center}
  \url{http://unige.ch/~hajian/gdg90/}
\end{center}
\subsection{Minimum and maximum eigenvalues of $B_i$}

Before performing convergence experiments on the Algorithm
\ref{algo:addschwarz} and \ref{algo:OS}, let us validate numerically
the asymptotic behavior of the minimum and maximum eigenvalues of the
operator $B_i$, i.e.~inequality (\ref{eq:Best}). To do so, we should
measure the minimum and maximum eigenvalues of $C_i^{} :=
\tilde{A}_\GAM^{-1/2} B_i^{} \tilde{A}_\GAM^{-1/2}$.  We generate a
sequence of quasi-uniform triangulations and construct the operators
$B_i$ and $\tilde{A}_{\GAM}$ for each triangulation. We denote the
size of each operator by $N$, i.e.~$B_i \in \RE^{N \times N}$. We have
$1/h \propto \sqrt{N}$ as $h$ goes to zero.

According to (\ref{eq:CestL}), the minimum eigenvalue of $C_i$ is
bounded from below independently of the mesh size. This can be seen
from Table \ref{tab:eigC}.
\begin{table}
  \caption{Minimum and maximum eigenvalues of $C_i$.}
  \label{tab:eigC}
  \centering
  \begin{tabular}{l|cccccc}
    $\sqrt{N}$ & 6 & 13 & 26 & 55 & 112 & 225
    \\
    \hline
    $\sigma_{\min}$ & 0.295 & 0.288 & 0.286 & 0.286 & 0.286 & 0.286
    \\
    $\sigma_{\max}$ & 0.335 & 0.415 & 0.457 & 0.478 & 0.489 & 0.494
   \end{tabular}
\end{table}
For the maximum eigenvalues of $C_i$, observe that $\sigma_{\max}$ is
less than $\frac{1}{2}$ and is increasing. In order to see the growth
rate we plot $\frac{1}{2} - \sigma_{\max}$ in Figure \ref{fig:eigC}
which decreases like $1/\sqrt{N} = O(h)$ as $N$ goes to infinity. 
\begin{figure}
  \centering
  \epsfig{file=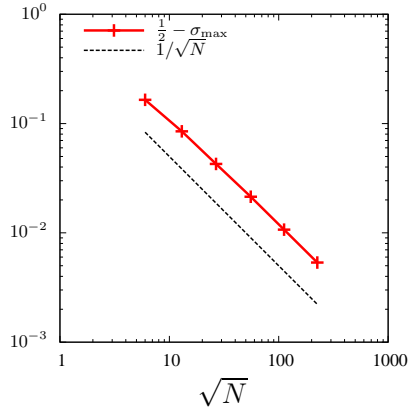, scale=0.75}
  \caption{Behavior of $(\frac{1}{2} - \sigma_{\max})$ versus total
    number of unknowns, $N$.}
  \label{fig:eigC}
\end{figure}
This is in agreement with (\ref{eq:CestU}).

\subsection{Two subdomain case}

In this section we compare the contraction factor of the two Schwarz
algorithms with respect to $h$-dependency. We perform both algorithms on
a sequence of unstructured meshes. We measure the number of iterations
required to reduce the relative error to $tol := 1\mbox{\sc{e}-}10$
while refining the mesh, that is
\begin{equation*}
  {\norm{ u_h^{(n)} - u_h }_{0}} \leq tol \, \norm{f}_{0}.
\end{equation*}
This level of accuracy is not necessary in practice since the error
between the exact and approximate solution, $\norm{u-u_h}_0$, is
much bigger and one usually can terminate the iteration after reaching
the accuracy level of the method. The domain is partitioned into two
by a non-straight interface; see Figure \ref{fig:ddmesh} (left).

As we see in the Figure \ref{fig:h-dep} (left), 
\begin{figure} \label{fig:h-dep}
  \centering
  \epsfig{file=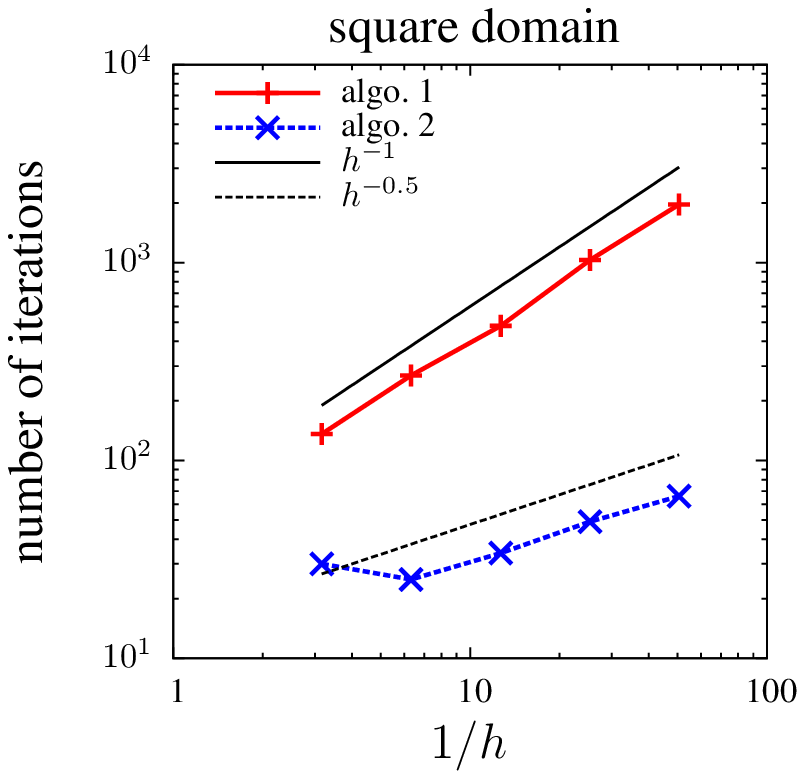, scale=0.75}
  \hspace{0.5cm}
  \epsfig{file=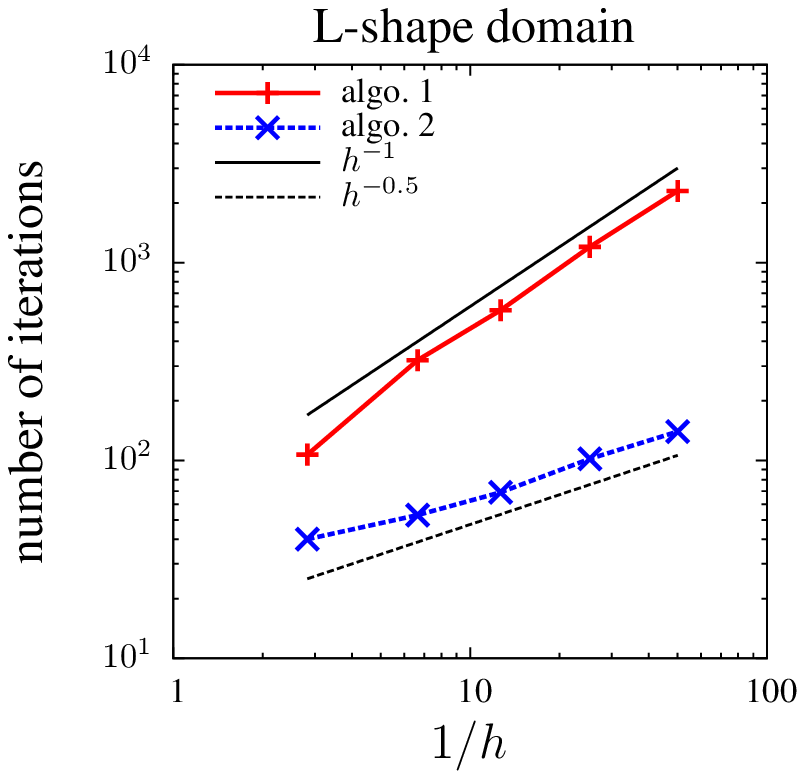, scale=0.75}	
  \caption{Convergence of Schwarz methods on a square domain
    (left) and L-shape domain (right).}
\end{figure}
on a square domain the number of iterations for Algorithm
\ref{algo:addschwarz} grows like $1/h$, which is equivalent to $\rho
\leq 1 - O(h)$, while for Algorithm \ref{algo:OS} it behaves like
$1/\sqrt{h}$, or in other words we have $\rho \leq 1 - O(\sqrt{h})$,
which illustrates well our analysis. This is the case for the L-shape
domain too, see Figure \ref{fig:h-dep} (right).

\subsection{Multi subdomains case}

We now show some numerical results on the multi subdomain
algorithm. The subdomains are formed by a coarse triangulation of the
domain which we call $\Ti{H}$. We consider a nested fine mesh and
therefore $\Ti{H} \subset \Th$. An example is given in Figure
\ref{fig:ddmesh} (right).  We consider here {\it four} subdomains
which share a cross-point, and similarly to the two subdomain case we
measure the number of iterations necessary to reach the desired
tolerance. We observe in Table \ref{tab:conv4sub}
\begin{table}
  \caption{Convergence of OSM for four subdomains} \centering
  \label{tab:conv4sub}
  \begin{tabular}{l|ccccc}
    Mesh size & $h_0$ & $h_0/2$ & $h_0/4$ & $h_0/8$ & $h_0/16$
    \\
    \hline
    \# iterations & 25 & 35 & 57 & 82  & 117
  \end{tabular}
\end{table}
that the contraction factor asymptotically is $\rho = 1 -
O(\sqrt{h})$, i.e.~$82/57\approx1.43$ or $117/82 \approx 1.42$ which
are close to $\sqrt{2}$.

\subsection{OSM as a preconditioner}

We use now the optimized Schwarz method as a preconditioner for GMRES
with the tolerance $tol := 1\mbox{\sc{e}-}6$. In order to provide a
qualitative comparison we also consider the widely used conjugate
gradient method with a one-level additive Schwarz preconditioner
applied to the original system (\ref{eq:linsys}).  We consider 16
subdomains illustrated in Figure \ref{fig:ddmesh} (right). We observe
in Table \ref{table:gmres2}
\begin{table} 
  \caption{Number of iterations required by OSM-GMRES and PCG
    to reach the desired tolerance.}
  \label{table:gmres2}
  \centering
  \begin{tabular}{l|ccccc}
    Mesh size & $h_0$ & $h_0/2$ & $h_0/4$ & $h_0/8$ & $h_0/16$ 
    \\
    \hline
    OSM-GMRES & 20 & 52  & 60 & 72 & 87
    \\
    PCG & 14 & 38 & 55 & 104 & 154 
  \end{tabular}
\end{table}
that the number of iterations for OSM-GMRES grows like
$O(h^{-1/4})$. This is because Krylov methods benefit often from
another square-root in their contraction factor compared to the
stationary iteration method. Therefore the contraction factor of
OSM-GMRES is $\rho = 1 - O(h^{1/4})$, i.e.~$72/60 \approx 1.2, 87/72
\approx 1.2$ which are close to $2^{1/4}$. For preconditioned
(additive Schwarz) conjugate gradient method, we have $\rho = 1 -
O(\sqrt{h})$.

We would like to comment on the size of the augmented system. In case
of mesh size $h_0/16$ we have 19,032 DOFs for the primal variable
$u_h$ and 1,296 DOFs for the interface unknowns. Therefore the
augmented system is very little changed in size compared to the
original system.

\section{Conclusion}

We have presented and analyzed classical and optimized Schwarz methods
for IPH discretizations. The interesting fact is that both use Robin
transmission conditions, but we proved that for an arbitrary
two-subdomain decomposition the classical Schwarz algorithm has a
convergence factor $1 - O(h)$, while the optimized one has a
contraction factor $1-O(\sqrt{h})$. This is because the IPH
discretization imposes a bad choice of the Robin parameter on the
method. We then generalized the definition of the algorithms to the
multi-subdomain case, and showed by numerical experiments that our
theoretical results still hold. We finally illustrated the potential
benefit that one obtains using OSM as a preconditioner compared to PCG.
\Appendix \label{appendix}
\section*{proof of inequalities for $\Zopt(\cdot)$}
In this part we provide some proofs regarding the extension by zero
operator, $\Zopt_i(\cdot)$.  First we recall inverse and mass matrix
inequalities; see \cite[Appendix B]{widlund} and references therein.
All constants are independent of $h$.  Let $w \in \PO^1(K)$ where $K$
is a simplex in $\RE^d$. Then the inverse inequality
\begin{equation} \label{eq:invineq}
  \norm{ \nabla w }_{K} \leq \frac{c}{h} \norm{ w }_{K}
\end{equation}
holds. Let $\BO{w}$ be the DOFs of $w$ and $M_d$ be the corresponding
mass matrix. Then we have
\begin{equation*}
  c_1 \, h^{d} \, \BO{w}^{\top} \BO{w} \leq \BO{w}^{\top} M_d \BO{w} \leq 
  c_2 \, h^{d} \, \BO{w}^{\top} \BO{w}.
\end{equation*}
\begin{lemma}
Let $\varphi \in \Lambda_h$ and $\Zopt_i(\varphi)$ be its extension by
zero operator into $\OM_i$.  For an element $K$ which
shares an edge with the interface, we have
\begin{equation*}
  \begin{array}{lcl}
    \norm{ \nabla \Zopt_i( \varphi ) }_{K}^{2} &\leq& {C_1}{h^{-1}}
    \norm{\varphi}_{e}^{2}, \\ \norm{ \Zopt_i( \varphi ) }_{K}^{2}
    &\leq& C_2 \, h \norm{\varphi}_{e}^{2}.
  \end{array}
\end{equation*}
\end{lemma}
\begin{proof}
  Let $\phiB_e := (\varphi_1, \varphi_2)$ be the DOFs of $\varphi$ on
  the edge shared with the interface.  Moreover let $w =
  \Zopt_i(\varphi)|_{K}$. Then we have $\BO{w} = (\varphi_1,
  \varphi_2, 0)$.  For the first inequality we invoke the inverse
  inequality.  Assuming the mesh is quasi-uniform, i.e.~$h_e \approx
  h_K \approx h$, we get
  \begin{equation*}	
    \begin{array}{rcl}
      \norm{ \nabla w }_{K}^{2} \leq \frac{c^2}{h^2} \norm{ w }^{2}_{K}
      &\leq& c_1 h^{d-2} ( \varphi_1^2 + \varphi_2^2 + 0)
      \\
      &\leq& c_2 h^{d-2} h_e^{-(d-1)} \phiB_e^\top M_{d-1} \phiB_e
      \\
      &\leq& c_3 h^{-1} \phiB_e^\top M_{d-1} \phiB_e
      \\
      &=& c_3 h^{-1} \norm{ \varphi }_{e}^{2}.
    \end{array}	
  \end{equation*}
  The proof for the second inequality follows the same steps. \quad
\end{proof}
\begin{lemma}
Let $\varphi \in \Lambda_h$ and $\Zopt_i(\varphi)$ be its extension by
zero operator into $\OM_i$.  Then
\begin{equation*}
  \norm{ \jump{ \Zopt_i( \varphi ) } }_{\EPS_i}^{2} \leq
  C \norm{ \varphi }_{\GAM}^{2},
\end{equation*}
where $C \geq 1$.
\end{lemma}
\begin{proof}
We start by those edges which are part of the interface, see Figure
\ref{fig:extzero}, e.g.~$e_1$ and $e_3$.  We have
$$
 \sum_{e \in \GAM} \norm{ \jump{ \Zopt_i( \varphi ) } }_{e}^2 = 
	\sum_{e \in \GAM} \norm{ \Zopt_i( \varphi ) }_{e}^{2}
	= \sum_{e \in \GAM} \norm{  \varphi }_{e}^2 = \norm{  \varphi }_{\GAM}^2,
$$ 
which shows already that $C \geq 1$. Consider those edges $e \in
\EPS_i$ that are not on the interface but belong to an element which
shares an edge with the interface, e.g.~$e^\ast := \partial K_1 \cap
\partial K_2$ in Figure \ref{fig:extzero}.  Let $\phiB_e :=
(\varphi_1, \varphi_2)$ be the DOFs of $\varphi$ on $e_2$ and assume
$\varphi_2$ is the DOF which is also located on $e^\ast$. Then we have
\begin{equation*}
  \norm{ \jump{ \Zopt_i( \varphi ) } }^{2}_{e^\ast} 
  = (\varphi_2,0) M_{d-1} (\varphi_2,0)^{\top}
  \leq c h_{e^{\ast}}^{d-1} \varphi_2^{2}
  \leq c h_{e^{\ast}}^{d-1} ( \varphi_2^{2} + \varphi_1^{2} )
  \leq c_1 \norm{ \varphi }_{e}^{2},
\end{equation*}
where we again used the quasi-uniformity of the mesh ($h_e \approx h
\approx h_{e^\ast}$). The other case would be $K_1$ and $K_3$ share an
edge, for which we can use the same argument.  For other edges $ \jump{
  \Zopt_i(\varphi) } $ is simply zero. \quad
\end{proof}
\bibliographystyle{siam}
\bibliography{paper}

\begin{thebibliography}{10}

\bibitem{blanca2}
{\sc Paola~F. Antonietti and Blanca Ayuso}, {\em Multiplicative schwarz methods
  for discontinuous galerkin approximations of elliptic problems}, ESAIM:
  Mathematical Modelling and Numerical Analysis, 42 (2008), pp.~443--469.

\bibitem{paola2011}
{\sc Paola~F Antonietti and Paul Houston}, {\em A class of domain decomposition
  preconditioners for hp-discontinuous galerkin finite element methods},
  Journal of Scientific Computing, 46 (2011), pp.~124--149.

\bibitem{dgunified}
{\sc Douglas~N. Arnold, Franco Brezzi, Bernardo Cockburn, and L.~Donatella
  Marini}, {\em Unified analysis of discontinuous {G}alerkin methods for
  elliptic problems}, SIAM J. Numer. Anal., 39 (2001/02), pp.~1749--1779.

\bibitem{brenner}
{\sc Susanne~C. Brenner}, {\em The condition number of the {S}chur complement
  in domain decomposition}, Numer. Math., 83 (1999), pp.~187--203.

\bibitem{brenner-poincare}
\leavevmode\vrule height 2pt depth -1.6pt width 23pt, {\em
  Poincar\'e-{F}riedrichs inequalities for piecewise {$H^1$} functions}, SIAM
  J. Numer. Anal., 41 (2003), pp.~306--324.

\bibitem{castillo}
{\sc Paul Castillo}, {\em Performance of discontinuous {G}alerkin methods for
  elliptic {PDE}s}, SIAM J. Sci. Comput., 24 (2002), pp.~524--547.

\bibitem{cockburn}
{\sc Bernardo Cockburn, Jayadeep Gopalakrishnan, and Raytcho Lazarov}, {\em
  Unified hybridization of discontinuous {G}alerkin, mixed, and continuous
  {G}alerkin methods for second order elliptic problems}, SIAM J. Numer. Anal.,
  47 (2009), pp.~1319--1365.

\bibitem{discacciati2004operator}
{\sc Marco Discacciati}, {\em An operator-splitting approach to nonoverlapping
  domain decomposition methods}, Rapport de la Section de Math{\'e}matiques,
  EPFL,  (2004).

\bibitem{dryja}
{\sc Maksymilian Dryja, Juan Galvis, and Marcus Sarkis}, {\em {BDDC} methods
  for discontinuous galerkin discretization of elliptic problems}, Journal of
  Complexity, 23 (2007), pp.~715 -- 739.
\newblock Festschrift for the 60th Birthday of Henryk Woźniakowski.

\bibitem{dolean}
{\sc Mohamed El~Bouajaji, Victorita Dolean, Martin~J Gander, Stephane Lanteri,
  Ronan Perrussel, et~al.}, {\em D{G} discretization of optimized {S}chwarz
  methods for {M}axwell's equations},  (2013).

\bibitem{ewing}
{\sc Richard~E. Ewing, Junping Wang, and Yongjun Yang}, {\em A stabilized
  discontinuous finite element method for elliptic problems}, Numer. Linear
  Algebra Appl., 10 (2003), pp.~83--104.
\newblock Dedicated to the 60th birthday of Raytcho Lazarov.

\bibitem{karakashian}
{\sc Xiaobing Feng and Ohannes~A. Karakashian}, {\em Two-level additive
  {S}chwarz methods for a discontinuous {G}alerkin approximation of second
  order elliptic problems}, SIAM J. Numer. Anal., 39 (2001), pp.~1343--1365
  (electronic).

\bibitem{ganderos}
{\sc Martin~J. Gander}, {\em Optimized {S}chwarz methods}, SIAM J. Numer.
  Anal., 44 (2006), pp.~699--731 (electronic).

\bibitem{gander2008schwarz}
{\sc Martin~J Gander}, {\em Schwarz methods over the course of time},
  Electronic Transactions on Numerical Analysis, 31 (2008), p.~5.

\bibitem{hajian2013block}
{\sc Martin~J. Gander and Soheil Hajian}, {\em Block {J}acobi for discontinuous
  {G}alerkin discretizations: no ordinary {S}chwarz methods}, Domain
  Decomposition Methods in Science and Engineering XXI, Lect. Notes Comput.
  Sci. Eng. Springer,  (2013).

\bibitem{gander2012best}
{\sc Martin~J Gander and Felix Kwok}, {\em Best {R}obin parameters for
  optimized {S}chwarz methods at cross points}, SIAM Journal on Scientific
  Computing, 34 (2012), pp.~A1849--A1879.

\bibitem{gander2013applicability}
\leavevmode\vrule height 2pt depth -1.6pt width 23pt, {\em On the applicability
  of {L}ions’ energy estimates in the analysis of discrete optimized
  {S}chwarz methods with cross points}, in Domain Decomposition Methods in
  Science and Engineering XX, Springer, 2013, pp.~475--483.

\bibitem{gander2014CPDD}
{\sc Martin~J Gander and K{\'e}vin Santugini}, {\em Cross-points in domain
  decomposition methods with a finite element discretization}, in preparation,
  (2014).

\bibitem{MZA:8194617}
{\sc Claude~J. Gittelson, Ralf Hiptmair, and Ilaria Perugia}, {\em Plane wave
  discontinuous galerkin methods: Analysis of the h-version}, ESAIM:
  Mathematical Modelling and Numerical Analysis, 43 (2009), pp.~297--331.

\bibitem{hajian2014}
{\sc Soheil Hajian}, {\em An optimized {S}chwarz algorithm for discontinuous
  {G}alerkin methods}, Domain Decomposition Methods in Science and Engineering
  XXII,  (2014).

\bibitem{lehrenfeld2010hybrid}
{\sc Christoph Lehrenfeld}, {\em Hybrid discontinuous {G}alerkin methods for
  incompressible flow problems}, master's thesis, RWTH Aachen, 2010.

\bibitem{lui}
{\sc S.~H. Lui}, {\em A {L}ions non-overlapping domain decomposition method for
  domains with an arbitrary interface}, IMA J. Numer. Anal., 29 (2009),
  pp.~332--349.

\bibitem{qin}
{\sc Lizhen Qin and Xuejun Xu}, {\em On a parallel {R}obin-type nonoverlapping
  domain decomposition method}, SIAM Journal on Numerical Analysis, 44 (2006),
  pp.~pp. 2539--2558.

\bibitem{joachim1}
{\sc Joachim Schöberl and Christoph Lehrenfeld}, {\em Domain decomposition
  preconditioning for high order hybrid discontinuous galerkin methods on
  tetrahedral meshes}, in Advanced Finite Element Methods and Applications,
  Thomas Apel and Olaf Steinbach, eds., vol.~66 of Lecture Notes in Applied and
  Computational Mechanics, Springer Berlin Heidelberg, 2013, pp.~27--56.

\bibitem{widlund}
{\sc Andrea Toselli and Olof Widlund}, {\em Domain decomposition
  methods---algorithms and theory}, vol.~34 of Springer Series in Computational
  Mathematics, Springer-Verlag, Berlin, 2005.

\bibitem{hesthaven}
{\sc T.~Warburton and J.~S. Hesthaven}, {\em On the constants in {$hp$}-finite
  element trace inverse inequalities}, Comput. Methods Appl. Mech. Engrg., 192
  (2003), pp.~2765--2773.

\end{thebibliography}
\end{document}